\documentclass[11pt]{amsart}
\usepackage{amssymb,amsfonts,amsthm,amsmath,enumerate,amscd,mathrsfs}
\usepackage[english]{babel}
\usepackage{colordvi,color}

\newtheorem{lemma}{Lemma}[section]
\newtheorem{claim}[lemma]{Claim}
\newtheorem{corollary}[lemma]{Corollary}
\newtheorem{definition}[lemma]{Definition}

\newtheorem{proposition}[lemma]{Proposition}
\newtheorem{remark}[lemma]{Remark}
\newtheorem{theorem}[lemma]{Theorem}
\newcommand{\field}[1]{\mathbb{#1}}
\newcommand{\scr}[1]{\mathscr{#1}}

\newcommand{\C}{\field{C}}
\newcommand{\N}{\field{N}}

\newcommand{\R}{\field{R}}
\newcommand{\Rn}{{\R^n}}
\newcommand{\Rp}{{\R^p}}
\newcommand{\Rs}{{\R^s}}
\newcommand{\Rnbar}{\overline{\Rn}}

\newcommand{\Cal}[1]{\mathcal{#1}}
\newcommand{\arel}{{\bf a}_{rel}}
\newcommand{\aph}{{\alpha}}
\newcommand{\ba}{{\bf a}}

\newcommand{\bB}{{\bf B}}

\newcommand{\bc}{{\bf c}}

\newcommand{\bF}{{\bf F}}
\newcommand{\bG}{{\bf G}}

\newcommand{\bI}{{\bf I}}
\newcommand{\bm}{{\bf m}}
\newcommand{\bM}{{\bf M}}
\newcommand{\bMbar}{{\overline{\bf M}}}
\newcommand{\bn}{{\bf n}}

\newcommand{\bbo}{{\bf 0}}
\newcommand{\bp}{{\bf p}}
\newcommand{\bP}{{\bf P}}

\newcommand{\br}{{\bf r}}
\newcommand{\bR}{{\bf R}}

\newcommand{\bS}{{\bf S}}
\newcommand{\bt}{{\bf t}}
\newcommand{\bu}{{\bf u}}
\newcommand{\bv}{{\bf v}}
\newcommand{\bV}{{\bf V}}
\newcommand{\bw}{{\bf w}}
\newcommand{\bx}{{\bf x}}
\newcommand{\by}{{\bf y}}
\newcommand{\bz}{{\bf z}}

\newcommand{\cH}{\Cal{H}}
\newcommand{\cL}{\Cal{L}}
\newcommand{\cM}{\Cal{M}}

\newcommand{\cT}{\Cal{T}}
\newcommand{\cU}{\Cal{U}}
\newcommand{\cV}{\Cal{V}}
\newcommand{\cW}{\Cal{W}}

\newcommand{\cY}{\Cal{Y}}
\newcommand{\clos}{{\rm clos}}
\newcommand{\crit}{{\rm crit}}
\newcommand{\dd}{{\partial}}
\newcommand{\dist}{{\rm dist}}
\newcommand{\dlt}{{\delta}}

\newcommand{\gm}{\gamma}
\newcommand{\Gm}{\Gamma}

\newcommand{\kp}{{\kappa}}
\newcommand{\la}{{\langle}}
\newcommand{\lbd}{{\lambda}}
\newcommand{\Lbd}{{\Lambda}}
\newcommand{\lk}{{\rm Lk}}

\newcommand{\Mbar}{\overline{M}}
\newcommand{\nb}{{\nabla}}

\newcommand{\omg}{\omega}
\newcommand{\Omg}{\Omega}

\newcommand{\ra}{{\rangle}}
\newcommand{\rd}{{\rm d}}
\newcommand{\Rgo}{{\R_{\geq 0}}}

\newcommand{\scrT}{\scr{T}}
\newcommand{\sgm}{\sigma}
\newcommand{\Sgm}{\Sigma}

\newcommand{\Tht}{{\Theta}}
\newcommand{\ve}{\varepsilon}
\newcommand{\vp}{\varphi}
\newcommand{\zt}{{\zeta}}
\numberwithin{equation}{section}

\begin{document}
\title{
Equi-singularity of real families and %%total
Lipschitz Killing curvature densities at infinity
}
\author{Nicolas Dutertre and Vincent Grandjean}

\address{Univ Angers, CNRS, LAREMA, SFR MATHSTIC, F-49000 Angers, France}
\email{nicolas.dutertre@univ-angers.fr}
\address{Departamento de Matem\'atica, 
Universidade Federal de Santa Catarina, 
88.040-900 Florianópolis - SC, Brasil,
Brasil}
\email{vincent.grandjean@ufsc.br}
%
%    General info
%\subjclass[2010]{Primary 14P10, Secondary 57R70 03C64}
%
%\date{\today}
%
%\dedicatory{This paper is dedicated to our authors.}
%
%\keywords{Gauss-Kronecker curvature, total curvatures, generalized critical 
%values, definable families}
%
\thanks{{$ $ \\Both authors were partially supported by the ANR project LISA 
17-CE400023-01}.\\
Nicolas Dutertre is partially supported by the Centre Henri Lebesgue, program ANR-11-LABX-0020-0.\\
Vincent Grandjean was supported by 
FUNCAP/CAPES/CNPq-Brazil grant  306119/2018-8.}

%\thanks{Mathematics Subject Classification (2010) : 14B05, 32C18, 58K45  \\
%N. Dutertre is supported by {\em Agence Nationale de la Recherche}
%(reference ANR-08-JCJC-0118-01)\\

\begin{abstract}
Fix an o-minimal structure expanding the ordered field of real numbers.
Let $(W_\by)_{\by\in\Rs}$ be a definable family of closed 
subsets of $\Rn$ whose total space $W = \cup_\by W_\by\times\by$ is 
a closed connected $C^2$ definable sub-manifold of $\Rn\times\Rs$.
Let $\vp:W \to\Rs$ be the restriction of the projection to the second
factor. 

After defining $K(\vp)$, the set of generalized critical values of $\vp$,
showing that they are closed and definable of positive codimension in $\Rs$,
contain the bifurcation values of $\vp$ and are stable under 
generic plane sections, we prove that all the Lipschitz-Killing curvature  
densities at infinity $\by \mapsto \kp_i^\infty(W_\by)$ are continuous 
functions over $\Rs\setminus K(\vp)$. When $W$ is a $C^2$ definable 
hypersurface of $\Rn\times\Rs$, we further obtain that the symmetric 
principal curvature densities at infinity $\by \mapsto \sgm_i^\infty(W_\by)$ 
are continuous functions over $\Rs\setminus K(\vp)$. 
\end{abstract}

\maketitle
\tableofcontents
%
%
%
%
%
%
%
%
%
%
%
%
%
%
%
%
%
%
%
%
%
%
%
%
%
%
%
%                          
%    *******************************************************************
%
%
%
%
%
%
%
%
%
%
%
%
%
%
%
%
%
%
%
%
\section{Introduction}
Let $(W_\by)_{\by\in P}$ be a family of subsets of $\Rn$ or $\C^n$ 
with parameter space $P$. This is the family of the (projections onto $\Rn$ 
or $\C^n$ of the) levels of 
the projection $\vp$ of the total space
of the family  $W:= \cup_{\by\in P} W_\by \times \by$ onto the parameter
space 
$$
\vp : W \to P.
$$

Assuming that $\vp$ is continuous, this family is
\em locally trivial at $\bc \in P$ \em if $\vp$ induces a trivial topological 
fibre bundle structure over a neighbourhood of $\bc$, with model fibre
$W_\bc$ for which $\vp$ is the projection-onto-the-base mapping. Similarly, 
a $C^k$ mapping $F:X\to P$, with $k\geq 1$, 
is \em locally trivial at $\bc \in P$ \em
if it induces a $C^{k-1}$ trivial fibre bundle structure over a neighbourhood
of $\bc$, with model fibre $F^{-1}(\bc)$. A proper and surjective $F$ induces
a locally trivial fibre bundle over $P$, by Ehresmann's Fibration Theorem
\cite{Ehr}. A value at which $F$ is not locally trivial is a 
\em bifurcation value. \em Critical values are bifurcation values.

Whenever the mapping or the family is rigid/regular in some explicit sense,
bifurcations values are rare, e.g. 
real and complex polynomial functions admits finitely 
many bifurcation values \cite{Tho}. Regular bifurcation values are even 
rarer and are hard to detect. When $X$ and $P$ are affine spaces, 
sufficient conditions for $F$ or $W$ to be locally trivial at a regular
value already exist. Among others are 
\em Malgrange-Rabier condition \cite{Pha,Rab,KOS}, $\bt$-regularity 
\cite{ST,Tib98,Tib99,DRT}, $\rho$-regularity \cite{NeZa,TLLZa,DRT}, 
spherical-ness at infinity \cite{dAGr1,dAGr2,DuGr1}. \em
Any such regularity condition at infinity at a given value compel the 
behaviour at infinity of the nearby levels to be "tame".

An interesting task, related to equi-singularity theory, is to find 
numerical criteria characterizing these regularity conditions at infinity.
For instance, following Teissier's works on plane sections and polar invariants 
\cite{TeissierCargese73,Teissier1977,TeissierLaRabida}, Tib\u{a}r produced, 
for each parameter $t$ of a complex family of affine hypersurfaces 
$(\{\bx\in\C^n : F(\bx,t) =0\})_{t\in\C}$, where $F:\C^n \times \C \mapsto
\C$ is a polynomial, a polar-like invariant vector $(\aph^*(t))$ with integer 
values \cite{Tib98}, whose local constancy at a value 
$c$ is equivalent to 
$\bt$-regularity nearby $c$. 
(See also \cite{HaLe,Par} for earlier special occurrences of equi-singularity 
numerical criteria for complex polynomials.)

Possible avatars, in the real setting, of the previously mentioned local 
(complex) polar invariants are integrals of Lipschitz-Killing curvatures. 
Since in general such integrals, as functions over the parameters 
of a given real family $W$, do not take only isolated values, we should  
instead look for the continuity of such functions. 
A first, yet essential, step consists of chasing sufficient conditions:
Given $\bR$ a regularity condition at infinity, can we produce 
a vector valued mapping, defined over the parameter space
of the family $W$, which would be continuous in a neighbourhood of a value 
$\bc$ once the family $W$ would be $\bR$-regular at $\bc$ ? 

Our approach for the results presented here feeds on two facts: 
(i) 
The stability of $\bt$-regularity of complex polynomial functions 
by generic plane sections \cite{Tib98}; (ii)
The general results about Lipschitz-Killings curvatures/measures of 
tame sets of the first named author \cite{DutertreGeoDedicata2004,DutertreAdvGeo,
DutertreGeoDedicata2012}
rely heavily on generic plane sections. 
In this paper, we produce a vector valued mapping, with components
built from Lipschitz-Killing curvatures/measures,  
and show it is continuous at any regular value which is also
Malgrange-Rabier regular.

\medskip
We treat the more general context of definable families of subsets of
$\Rn$ whose total space is a closed sub-manifold. It is sufficient to
work with a connected total space, which we assume. Let be given an 
o-minimal structure expanding the ordered field of real numbers. 
Consider the following $C^2$ definable mapping: 
$$
\vp : W \to \Rs, \, (\bx,\by) \mapsto \by
$$
defined over the definable closed connected $C^2$ sub-manifold $W$
of $\Rn\times\Rs$ with $\dim W \geq s$.  Since, by definition  
$\vp^{-1}(\by) = W_\by \times\by$, we are interested in the local triviality 
properties of the 
definable family $(W_\by)_{\by\in\Rs}$ of closed subsets of $\Rn$ nearby 
regular values (of $\vp$). Passing to the graph of the mapping $F$, the 
looked for properties for $F$ are in a simple and explicit correspondence with 
those of $\vp$, as checked here. The choice of $C^2$ is arbitrary. Everything
we will do have analogues in $C^k$-regularity, $2\leq k \leq \infty$.

The results of this  paper are 
of two sorts: I) those about Malgrange-Rabier regularity condition (see 
Definition \ref{def:Malgrange-Rabier}) later on shortened as (MR), and 
II) those about Lipschitz-Killing curvature densities at infinity. 

\smallskip\noindent
I) It is simpler to describe (MR)-regularity in the case of function (i.e. 
$s=1$): \em A value $c$ is (MR)-regular for $\vp$ if there exists a positive 
constant
$M_c$ such that for any sequence $(\bw_k)_k$ of $W$ such that
(i) $|\bw_k|\to \infty$, (ii) $\vp(\bw_k) \to c$, then \em
$$
k \gg 1 \; \Longrightarrow \; |\bw_k|\cdot |\nb \vp(\bw_k)| \geq M_c.
$$ 
The set of critical values of $\vp$ is $K_0(\vp)$ and that of non (MR)-regular 
values is $K_\infty(\vp)$. We shall proof in this most general 
context the following results:

\medskip\noindent
{\bf Theorem I.} \em 1) $K_\infty(\vp)$ is definable;
\\
2) $K(\vp) := K_0(\vp) \cup K_\infty(\vp)$ is closed, definable
(Lemma \ref{lem:Kf-def-closed}) and of positive codimension 
(Theorem \ref{thm:morse-sard}); 
\\
3) the family $(W_\by)_{\by\notin K(\vp)}$ is $C^1$ locally trivial at each of 
its point (Theorem \ref{thm:trivialisation}), from which we deduce
$Bif(\vp) \subset K(\vp)$ (Corollary \ref{cor:bif-kf});
\\
4) (MR)-condition at a regular value is stable by generic plane sections
(Theorem \ref{thm:M-R-plane-sections}). 
\em 
  
\medskip
If all known occurrences of statements 1), 2) and 3) of Theorem I are 
special cases of the present setting, point 4) is new in the real setting, 
and outside the case of complex polynomial functions.

\medskip\noindent
II) Let $Z$ be a closed connected $C^2$ sub-manifold of $\Rn$ of dimension $d$ 
(equipped with the restriction of the Euclidean metric tensor). 
The other objects we will consider are the Lipschitz-Killing curvature
densities at infinity of $Z$ defined as
$$
\kp_i^\infty(Z) := 
\lim_{R \to +\infty} \frac{1}{R^{d-i}} \int_{Z \cap \bB_R^n} 
K_i(Z,\bx) \rd\bx,
\;\; {\rm for} \;\;
i=0,\ldots,d,
$$
where $K_i(Z,\bx)$ is the $i$-th Lipschitz-Killing curvature of $Z$ at 
$\bx$. Note that $K_i(Z,\bx)=0$ if $i$ is odd or $\geq d+1$.
When $Z$ is the hypersurface $\{f=0\}$, oriented by $-\nb f|_Z$, let
$\sgm_i^Z(\bx)$ be the $i$-th symmetric function of the 
principal curvatures of $Z$ at $\bx$.
We further
define the symmetric principal curvature densities at infinity of $Z$ as
$$
\sgm_i^\infty (Z) := \lim_{R \to +\infty} \frac{1}{R^{n-1-i}}
\int_{Z \cap  \bB_R^n} \sigma_i^Z (\bx) \rd\bx,
\;\; 0 \leq i\leq n-1 .
$$
When $i$ is even $2\sgm_i^\infty (Z)= \kp_i^\infty(Z)$, but when $i$ is odd 
in general $\sgm_i^\infty(Z)$ is not identically null.
Considering Corollary \ref{cor:trivialisation}, 
we can reduce to the case where $W_\by$ is connected
for each $\by\notin K(\vp)$.  
A concise way to present our main results  Theorem \ref{ContCurvaturesMappings}
and Theorem \ref{thm:main-hyper} 
is as follows:

\medskip\noindent
{\bf Theorem II.} \em Assume that $W_\by$ is connected at each $\by \notin 
K(\vp)$. For each $i=0,\ldots,\dim W - s,$ the function $\by \mapsto 
\kp_i^\infty(W_\by)$ is continuous over $\Rs\setminus K(\varphi)$.
When the $W_\by$'s are furthermore hypersurfaces, for each $i=0,\ldots,n-1$,
the function $\by \mapsto \sgm_i^\infty(W_\by)$ is continuous over
$\Rs\setminus K(\vp)$.
\em

\medskip
Theorem II is the continuation of the results of the second named author 
\cite{GrandjeanBullBraz}, stating that the following total Gauss-Kronecker 
curvature functions, associated with a $C^2$ definable function 
$f:\Rn\to\R$, have at most finitely many discontinuities
$$
y \mapsto GK(y) := \int_{f^{-1}(y)} \kp_y(\bx)\rd\bx,\;\; {\rm and
} \;\;
y \mapsto |GK|(y) := \int_{f^{-1}(y)} |\kp_y(\bx)|\rd\bx,
\;\; {\rm where} \;\; \kp_y := \sgm_{n-1}^{f^{-1}(y)},
$$
and of our recent result \cite{DuGr1}, where continuity of 
$GK$ and $|GK|$ is proved at any regular value at which the function is
also spherically regular at infinity. It is worth mentioning here the recent 
preprint \cite{DihnPham} investigating 
sufficient conditions of equi-singularity at infinity in terms of topological
properties of set valued mappings over the levels of the given mapping $F$. 
Theorems I and II are to be associated with the results of 
\cite{ComteAnnENS2000,ComteMerleAnnENS2008,NguyenValette} about the continuity 
of local Lipschitz-Killing invariants along Whitney and Verdier strata of 
definable subsets.

\medskip
The paper is organized as follows: Section \ref{section:miscellaneous} 
introduces notations and a few reminders used in what follows. 
Section \ref{ref:o-min} lists a few facts about definability in order to
provide some exhaustive-ness. 
Section \ref{section:linear-algebra} to Section \ref{section:bif} present 
material for  Theorem I. We have given detailed proofs of facts 1) to 3). Our 
treatment uses heavily the Rabier number (Definition \ref{def:rabier}), and 
Lemma \ref{lem:nu-projection} explains why dealing with families includes 
the case of mappings. The proof of Theorem \ref{thm:trivialisation}
following Rabier's point of view \cite{Rab} is of interest in its own, and so 
is Remark \ref{rmk:flow-analytic}. 
Section \ref{section:sub-levels} will briefly deal with a trivialisation
result of the sub-level family associated with a family of hypersurfaces
(Proposition \ref{prop:trivial-sub-level}). 
Section \ref{section:link-infty} reviews in an explicit fashion 
properties on the family of links at infinity, which is a central 
tool to deal with the Lipschitz-Killing measures/curvatures.
We give also proofs of Proposition \ref{prop:link-infty-const}
and of the lesser known Proposition \ref{prop:link-infty-const-sub-level}.
Section \ref{section:plane-section-malgrange} takes care of 
point 4) of Theorem I. Theorem \ref{thm:M-R-plane-sections} is new and is 
one of the two cornerstones over which will be established Theorem II.
Section \ref{section:LKM} to Section \ref{section:cont-curv-hyper}
are devoted to present all the necessary material around Lipschitz-Killing 
measures/curvatures to obtain Theorem II. The second cornerstone needed
to produce Theorem II is the Gauss-Bonnet type Theorem 
\ref{thm:GaussBonnet} (from the 
first author \cite{DutertreAdvGeo,DutertreGeoDedicata2012}).

%
%
%
%
%
%
%
%
%
%
%
%
%
%
%
%
%
%
%
%
%
%
%
%
%
%
%
%
%                          
%           *********************************************************
%
%
%
%
%
%
%
%
%
%
%
%
%
%
%
%
%
%
%
%
\section{Miscellaneous}\label{section:miscellaneous}
Let $\Rgo := [0,+\infty[$ and $\R_{>0} := ]0,+\infty[$.

We denote by $\bbo$ the origin (or the null vector) of any vector space or 
subspace of dimension at least two.

The Euclidean unit closed ball of $\R^q$ centred at the origin is $\bB^q$.
If $\bw$ is a point of $\R^q$, the Euclidean closed ball of radius $r$ and
centre $\bw$ is $\bB^q (\bw,r)$. When $\bw$ is just the origin we write
simply $\bB_r^q$.

The Euclidean unit sphere of $\R^q$ centred at the origin is $\bS^{q-1}$.
The Euclidean sphere of radius $r$ centred at the origin is $\bS_r^{q-1}$.

The Grassmann manifold of vector $k$-planes of $\R^q$ is $\bG(k,q)$. 
We will sometimes write $\bP^{q-1}$ for $\bG(1,q)$. 
We will use the same notation $P$ to mean either a point of $\bG(k,q)$, or the 
corresponding vector subspace of $\R^q$. 

Any real vector space $\R^q$ comes equipped with the Euclidean metric, and
associated scalar product $\la-,-\ra$ and norm $|-|$.
Any vector subspace $E$ of $\R^q$ turns into an Euclidean space when restricting
the Euclidean structure. The unit sphere of $E$ is $\bS(E)$ and the orthogonal
of $E$ in $\R^q$ is $E^\perp$.

Given a sequence $(\bw_k)_k$ of $\R^q$, we will write $\bw_k \to \infty$ to
mean that $|\bw_k| \to +\infty$ as $k$ goes to $+\infty$. 

Given any $C^1$ sub-manifold $S$ of the Euclidean space $\R^q$, the 
tangent bundle $TS$ is a sub-vector bundle of $T\R^q|_S$, therefore is equipped 
with the restriction to $TS$ of the Riemannian metric tensor over $T\R^q|_S$. 

Let $X,Y$ be two connected $C^1$ sub-manifolds of $\R^q$.  
The pair $(X,Y)$ is  \em Whitney $(a)$-regular at the point $\by$ of $Y$, \em 
if (i) $Y$ satisfies \em the frontier condition \em $Y \subset \clos(X)\setminus X$,
and (ii) the following condition holds true: 
\textit{
For any sequence $(\bx_k)_k$ of $X$ such that $\bx_k \to \by$
and $T_{\bx_k} X \, \to \, P \, \in \, \bG(\dim X,q)$, we have that
$\; T_\by Y \, \subset \, P$.
}
%%$$
%%\forall \, (\bx_k)_k \, \in \, X \; {\rm such} \; {\rm that} \; \bx_k \to \by, 
%%\; {\rm and} \; T_{\bx_k} X \, \to \, P \, \in \, \bG(\dim X,q), \;
%%{\rm then} \; T_\by Y \, \subset \, P.
%%$$
%
%
Let $f:X\cup Y \mapsto \R$ be a continuous function such that each 
restriction 
$f|_X$ and $f|_Y$ is $C^1$. The function $f$ satisfies \em Thom condition 
$(\arel)$ at the point $\by$ of $Y$ \em (also called \em Thom 
$(\arel)$-regular\em) if the pair $(X,Y)$ is Whitney 
$(a)$-regular at $\by$ and the following conditions are verified:
(i) the functions $f|_X$ and $f|_Y$ have constant ranks, and 
(ii) 
\textit{
For any sequence $(\bx_k)_k$ of $X$ such that $\bx_k \to \by$ and 
$\ker D_{\bx_k} (f|_X)  \to  T \, \in \, \bG(\dim X-1,q)$, we have that 
that $\ker D_\by (f|_Y)  \subset  T$.
}
%%$$
%%{\rm (ii)} \;
%%\forall \, (\bx_k)_k \in  X \; {\rm such} \; {\rm that} \; \bx_k \to \by, 
%%\; {\rm and} \; \ker D_{\bx_k} (f|_X)  \to  T \, \in \, \bG(\dim X-1,q), \;
%%{\rm then} \; \ker D_\by (f|_Y)  \subset  T.
%%$$

Let $a\in \overline\R := [-\infty,+\infty]$. Let $f,g:(I,a) \to \R$ be two 
germs of real functions at $a$, where $I$ is any interval with non-empty interior
whose frontier in $\overline\R$ contains $a$. 
We  use the following notations:
$$
{\rm (i)} \; f \, \sim \, g \; \Longleftrightarrow \; \lim_a \frac{f}{g} 
\in \R^*, \; \; \; 
{\rm (ii)}  \; f \, \simeq \, g \; \Longleftrightarrow \; \lim_a \frac{f}{g} =1,
 \; \; \; \;
{\rm (iii)}  \; |f| \, \ll \, |g| \; \Longleftrightarrow \; 
\lim_a \frac{|f|}{|g|} = 0.
$$
If $f,g :(I,a) \to \R^s$, for $s\geq 2$, we will write 
$$
{\rm (iv)} \; f \, \sim \, g \; \Longleftrightarrow \; |f| \sim |g|, 
\; \; \; \; 
{\rm (v)}  \; f \, \simeq \, g \; \Longleftrightarrow \; |f| \simeq |g|.
$$
%
%
%
%
%
%
%
%
%
%
%
%
%
%
%
%
%
%
%
%
%
%
%
%
%
%
%
%
%
%                          
%          *****************************************************
%
%
%
%
%
%
%
%
%
%
%
%
%
%
%
\section{On o-minimal structures}\label{ref:o-min}
We collect here a very few definitions and facts about o-minimal structures
and definability we will use in the sequel of this paper 
(see \cite{vdD,vdDM} for a proper treatment). We adopt a point of view 
close to those of \cite{Kur,dAc}. 

\medskip
An \em o-minimal structure $\cM$ expanding the ordered field of real 
numbers $(\R,+,\cdot,>)$ \em  is a collection $(\cM_q)_{q \in \N}$, where each 
$\cM_q$ is a family of subsets of $\R^q$ satisfying the following axioms:

\smallskip
\noindent
1) For each $q\in \N$, the family $\cM_q$ is a boolean sub-algebra 
of subsets of $\R^q$. 
\\
2) For any pair of subsets $A \in \cM_p$ and $B \in \cM_q$, 
then $A \times B \in \cM_{p+q}$. 
\\
3) Let $\pi: \R^{q+1} \to \R^q$ be the projection on the first $q$ factors. 
Given any subset $A$ of $\cM_{q+1}$, its projection $\pi (A)$ is a subset 
lying in $\cM_q$.
\\
4) The algebraic subsets of $\R^q$ belong to $\cM_q$. \\
5) The family $\cM_1$ consists exactly of the finite unions of points and 
intervals.

Those points imply that the smallest o-minimal structure is the 
structure of the semi-algebraic subsets, thus contained in any other one.

\medskip
Assume that such an o-minimal structure $\cM$ is given for the rest of this 
paper. 

\medskip
A subset $A$ of $\R^q$ is a \em definable subset \em shortened as \em 
definable \em(\em{in the given o-minimal structure}\em), if $A \in \cM_q$. 

Let $A$ be a subset of $\R^q$, a mapping $A \mapsto \R^r$ is \em definable \em 
if its graph is a definable subset of $\R^q\times\R^r$. 

\medskip
Let $A$ be a definable subset of $\R^q$ and let $B$ be a 
definable subset of $A$. The pair $(A,B)$ admits a  
\em definable Whitney $(a)$-regular $C^k$ stratification, \em 
namely the following result holds true:

\smallskip\noindent
{\bf Theorem.} (see for instance \cite{Whi,Loj,DeSt,vdDM,NTT}).
\em
For each positive integer $k$, there exists a finite partition of $A$ 
into connected $C^k$ sub-manifolds, called strata with the following 
properties: (i) Each stratum is definable;
(ii) $B$ is a union of strata; (iii) Each pair $(X,Y)$ such that $\dim X > 
\dim Y$ either verifies the frontier condition $Y \subset \clos(X)\setminus 
X$, or the intersection $Y\cap \clos(X)$ is empty; (iv) Each pair of strata 
$(X,Y)$ satisfying the frontier 
condition is Whitney $(a)$-regular at each point of $Y$. \em

\medskip
Let $f:A\to \R$ be a  definable mapping. Let $B$ be a  
definable subset of $A$. The function $f$ admits a 
\em definable $(\arel)$-regular $C^k$ stratification, \em
that is the following statement holds true:

\smallskip\noindent
{\bf Theorem.} (see for instance \cite{Hir,Bek,KuPa,KuRa,Loi2,Loi3}).
\em
For each positive integer $k$, there exists a Whitney $(a)$-regular $C^k$ 
stratification of the pair $(A,B)$ satisfying the following additional property:
(v) For each pair of strata $(X,Y)$, the restriction $f|_{X\cup Y}$ is 
$(\arel)$-regular at every point of $Y$.
\em

\medskip
The notions of definability of subsets or mappings 
germ-ify along any definable subset. 
In particular, the germ of a mapping
$(\R,+\infty) \to \R^s$ is \em definable at $+\infty$ \em if it admits 
a representative which is  definable. 

Let $I,a$ as above.
Let $I^+$ be $(I \cap ]a,+\infty[,a)$ if $a<+\infty$, or $(\R,+\infty)$ when
and $a =+\infty$. Similarly let $I^-$ be $(I\cap ]-\infty,a[,a)$ when 
$a>-\infty$, or $(\R,-\infty)$ if $a = -\infty$. 

We would like to recall the following well known two elementary facts: \em 
Let $\ve \in \{+,-\}$ be such that $I^\ve$ is not empty. Let 
$f:(I^\ve,a) \to \R$ be definable. 

(i) For every non-negative integer number $k$, there 
exists $\cU_k$ an open neighbourhood of $a$ such that $f$ admits a 
$C^k$ representative $I^\ve\cap\cU_k \mapsto \R$.

(ii) There exists $\cU$ an open neighbourhood of $a$ such that $f$ admits a 
 monotonic representative $I^\ve\cap\cU \mapsto \R$.
\em

\medskip
Let $t\mapsto \gm(t) \in \Rn$ be a mapping definable at $+\infty$. 
The differential mapping $t\mapsto\gm'(t) \in\Rn$ is also definable at 
$+\infty$. We define 
$$
\bu(t) :=\frac{\gm(t)}{|\gm(t)|} \in \bS^{n-1}, \; {\rm and} \; 
L(t) := \R\gm'(t) \in \bP^{n-1}
$$ 
to be respectively the secant mapping and tangent direction mapping associated 
with $\gm$, both definable at $+\infty$. 
Definability at $+\infty$ guarantees that the following limits exist
$$
\bu_\gm := \lim_{t\to+\infty} \bu(t), \; {\rm and} \; 
L_\gm := \lim_{t\to+\infty} L(t).
$$
Moreover the following very useful result holds true. 
\begin{lemma}\label{lem:secant-tangent}\cite[Lemmas 2.5-7]{dAc}
1) $L_\gm = \R\bu_\gm$; 
\\
2) Assume furthermore that $|\gamma(t)| = t$ for $t$ large enough. Then
$$
\lim_{t\to+\infty} \gm'(t) = \bu_\gm.
$$
In particular $|\gm'|$ is bounded.
\end{lemma}
We also recall the following:
\begin{lemma}\label{lem:derivative-infinity}
Let $f :(\Rgo,+\infty) \to\R_{>0}$ be the germ of a function definable at 
$+\infty$ such that $f \to 0$ as $t \to +\infty$.
The following limits hold true
$$
(i) \, \lim_{t\to +\infty} t\cdot f'(t) = 0, \: {\rm and} \; (ii) \, 
\lim_{t\to+\infty} t^{1+\tau}\cdot \frac{f'(t)}{f(t)} = - \infty, 
\; \forall \, \tau >0.
$$
\end{lemma}
\begin{proof}
We recall that ultimately a function definable at $+\infty$ is monotonic.
In both cases an integrability argument near $+\infty$ combined with 
contradicting the desired statement yields the proof.
\end{proof}

\medskip\noindent
{\bf Our Setting:}
Let $W$ be a $C^2$ connected and closed sub-manifold of $\R^n\times\R^s$,
 definable in $\cM$. 
The restriction of the Euclidean metric onto $TW$ induces a  definable 
$C^1$ Riemannian metric over $W$.
We further require that
\begin{center}
$\dim W \geq s .$
\end{center}
Let $\vp : W\mapsto\R^s$ be the restriction to $W$ of the projection 
on $\R^s$, thus $\vp$ is a $C^2$ and  definable mapping.

%
%
%
%
%
%
%
%
%
%
%
%
%
%
%
%
%
%
%
%
%    *******************************************************************
%
%
%
%
%
%
%
%
%
%
%
%
%
%
%
%
%
%
\section{Some elementary linear algebra}\label{section:linear-algebra}
Let $V$ be a real vector space of dimension $q$.

Let $\bG$ be the union $\cup_k \bG(k,V)$, the total Grassmann space of $V$.

We assume that $V$ is equipped with a scalar product $\la-,-\ra$. Let $|-|$
be the associated norm. Any vector subspace $E$ of $V$ is equipped with
the restriction of the scalar product, to yield a scalar product on $E$. Let 
$\bS(E)$ be the unit sphere of $E$.

Let $E$ be a vector subspace of $V$ and let $\bu$ be a unit vector of $V$.
\em The distance of $\bu$ to $E$ \em is defined as
$$
\dlt_q(\bu,E) := \max\{|\la \bu,\bv\ra| \, : \, \bv \in \bS(E^\perp)\}.
$$
This function $(\bu,E)\mapsto \dlt_q(\bu,E)$ is semi-algebraic over 
$\bS(V)\times\bG$. 
\\
Let $E,F$ be vector subspaces of $V$.  \em The distance from $E$ to $F$ \em
is defined as $$
\dlt_q(E,F) := \max \{\dlt_q(\bu,F) : \bu\in\bS(E)\}.
$$
The function $(E,F) \to \dlt_q(E,F)$ is also semi-algebraic in its entries in 
$\bG\times\bG$. It is not symmetric in $E$ and $F$. Yet its restriction to 
$\bG(k,V)\times\bG(k,V)$ yields a distance on $\bG(k,V)$.

\medskip
Let $\cL(p,q)$ be the space of linear mappings from $\Rp$ to $\R^q$.
It is equipped with its Euclidean norm.

Let $\Sigma(p,q)$ be the algebraic subset of $\cL(p,q)$ consisting of all 
linear maps that are not surjective. 

For $A$ a linear operator of $\cL(p,q)$, let $A^* \in 
\cL(q,p)$ be the adjoint operator of $A$. 
\begin{definition}\label{def:rabier}
Let $A \in \cL(p,q)$. The \em Rabier number of the operator $A$ \em is 
defined as 
\begin{equation}\label{eq:rabier-def}
\nu(A) := \inf_{|\vp| = 1} |A^*(\vp)|. 
\end{equation}
\end{definition}
As can be seen in \cite{Rab,KOS,Jel1}, we know the following result:
\begin{proposition}\label{prop:rabier-equiv}
Let $A\in\cL(p,q)$. We have
$$
\inf_{|\vp| = 1} |A^*(\vp)| =
\sup\{r>0 \, : \, \bB^q(\bbo,r) \subset A\cdot \bB^p(\bbo,1)\}
=  \dist(A,\Sgm(p,q)).
$$
\end{proposition}
Let $A$ be any linear map of $\cL(p,q)$. 
Whenever the operator $A$ is not surjective, we get $\nu(A) = 0$. 
For any real number $\lbd$ we also get 
$$
\nu(\lbd A) = |\lbd| \nu(A).
$$
Let $N(A)$ be the orthogonal space of $\ker A$ in $\Rp$.
Although obvious it is worth singling out the following
\begin{equation}\label{eq:rabier-normal}
\nu(A) = \nu(A|_{N(A)}) \geq \nu(A|_P)
\end{equation}
for every subspace $P$ of $\Rp$. 
\\
In order to work with families rather than mappings, 
the following elementary result is key to this approach.
\begin{lemma}\label{lem:nu-projection}
Let $A \in \cL(p,q)$ and let $V$ be the graph of $A$, subspace of 
$\Rp\times\R^q$. Let $\vp:V \to\R^q$ be the restriction to $V$ of the 
canonical projection $\Rp\times\R^q \to \R^q$.

\smallskip\noindent
1) If $\nu(A)  >0$, then 
\begin{equation}\label{eq:nuA-SNA}
\nu(A) := \min \{|A\cdot\bu| : \bu \in \bS(N(A))\}.
\end{equation}

\noindent
2) The following is true
\begin{equation}\label{eq:nu-projection}
\nu(\vp) =  \frac{\nu(A)}{\sqrt{1+\nu(A)^2}}.
\end{equation} 
\end{lemma}
\begin{proof} 
Point 1) is straightforward.

\smallskip\noindent
For point 2), observe that $\nu(\vp) = 0$ if and only if 
$\nu(A) =0$. 
\\
We thus assume that $\nu(A)$ is positive.
Working with $N(A)$ instead of $\Rp$ does not change
the value of  $\nu(A)$. We can thus assume that 
$V$ is the graph of $A|_{N(A)}$.

\medskip
Let us consider the following real algebraic subset of $N(A)$
$$
S_A := \{\bu \in N(A) : |\bu|^2 + |A\cdot \bu|^2 = 1 \}.
$$
The subset $S_A$ is compact and does not contain $\bbo$.
Consider the following continuous semi-algebraic mapping
$$
a : S_A \to \R, \;\; \bu \mapsto \frac{|A\cdot\bu|}{|\bu|} .
$$
Since the radial projection of $S_A$ over $\bS(N(A))$ is a homeomorphism
,
the image of $a$ is the interval $[\nu(A),\Vert A\Vert]$.
For $\bu\in S_A$ we deduce that
$$
|\bu| = \frac{1}{\sqrt{1+a(\bu)^2}},\;\; {\rm and} \;\; |A\bu| = 
\frac{a(\bu)}{\sqrt{1+a(\bu)^2}}.
$$
By definitions of $\vp$ and of $\nu(\vp)$ we get 
$$
\nu(\vp) = \min \{|\bv| : (\bu,\bv) \in \bS(V)\} = \min\{|A\cdot\bu| : 
\bu \in S_A\}.
$$
Since the function $t \mapsto \frac{t}{\sqrt{1+t^2}}$ is increasing, we 
obtain 
$$
\nu(\vp) = \frac{\nu(A)}{\sqrt{1+\nu(A)^2}}.
$$
\end{proof}

%
%
%
%
%
%
%
%
%
%
%
%
%
%
%
%
%
%
%
%                          
%        **************************************************************
%
%
%
%
%
%
%
%
%
%
%
%
%
%
%

\section{Blowing-up at infinity}\label{section:t-reg}
We need to 
describe conveniently in our setting (and for our purposes) a neighbourhood 
of infinity.

\medskip
Let $\bI:= ]0,+\infty[$.
Let $\bM := \bI\times \bS^{n-1}$. 
Given the point $\bm =(\br,\bu) \in \bI \times \bS^{n-1}$, we observe that
$T_\bm \bM = \R\times T_\bu\bS^{n-1}$. 

\medskip
The spherical blowing-up of $\R^n$ at infinity 
is the following Nash diffeomorphism, defined as 
$$
\beta \, :  \, \bM \, \to \, \Rn\setminus \bbo, \;\;  \bm  \mapsto \bx = 
\frac{\bu}{\br}.
$$
The mapping $\beta$ provides a single chart to investigate the behaviour of 
any mapping, with source a non-bounded domain
of $\Rn$, nearby the boundary part  of the domain lying at infinity.

\medskip
Let $\bMbar := [0,+\infty[\times \bS^{n-1}$. It is a semi-algebraic subset of 
$\R\times\R^n$ and is a sub-manifold, with smooth and algebraic boundary 
$$
\bM^\infty := 0\times\bS^{n-1}.
$$
We compactify $\Rn$ as the Nash manifold with boundary 
$\Rnbar :=\bMbar \sqcup \bbo = \Rn\sqcup
\bS^{n-1}$ identifying $\Rn\setminus\bbo$ with $\bM$ via the mapping $\beta$.
We furthermore define 
$$
\bM_s := \bM \times \Rs, \;\; \bMbar_s:=\bMbar \times\Rs, \;\;
{\rm and} \;\;\bM_s^\infty := \bM^\infty\times\Rs.
$$
Let $\bG(k,\bMbar_s)$ be the Grassmann bundle of subspaces of rank $k$ of 
$T\bMbar_s$.
\\
Let $\clos(Z)$ be the closure in $\bMbar_s$ of the $C^1$ connected 
definable sub-manifold $Z$ of $\bM_s$. Let 
$$
Z^\infty := \clos(Z) \cap \bM_s^\infty.
$$
Let $g:Z \to \R$ be a definable $C^1$ function. We define the 
\em relative tangent bundle of $g$ over $\clos(Z)$ \em as follows
$$
\cT_g  :=\clos \{(\bz,T_\bz g ) \in 
\bG (\dim Z-1,\bMbar_s)|_{Z\setminus  \crit(g)} \} \subset \bG(\dim Z -1,
\bMbar_s)|_{\clos(Z)},
$$
where $T_\bz g$ is the tangent space to the level of $g$ through the point $\bz$.
Note that $\cT_g$ is closed and definable in $\bG (\dim Z - 1,\bMbar_s)$. 
Let $\pi : \bG(\dim Z - 1,\bMbar_s) \to \bMbar_s$ be the projection $(\bz,T)
 \mapsto \bz$ on the base of the Grassmann bundle. 
The fibre $\cT_{g,\bz}$ of $\cT_g$ over $\bz \in \bMbar_s$ is 
the subset of $\bG(\dim Z -1,T_\bz\bMbar_s)$ defined as
$$
\bz \times \cT_{g,\bz} := \cT_g \cap \pi^{-1}(\bz).
$$
\begin{definition}
The \em relative tangent bundle of $g$ over $\clos(Z)$ at infinity \em is 
defined as 
$$
\cT_g^\infty := \pi^{-1}(Z^\infty) \cap \cT_g.
$$
For any $\bz \in Z^\infty$, let $(\cT_g^\infty)_\bz$ be 
the fibre of $\cT_g$ above $\bz$.
\end{definition}

Let $\br_Z :Z\to \R$ be the $C^1$ definable function given by the 
restriction of $\br$ to $Z$, that is for $\bz = (\br,\bu,\by) \in \bMbar_s$ 
we find $\br_Z(\bz) = \br$. 

%%\medskip
%%Let $\tau:\bMbar_s \to \Rs$ be the projection onto the second factor $\Rs$. 

\medskip
The behaviour nearby $Z^\infty$ of the function $\br_Z$ contains some
information about the accumulation of $Z$ onto $Z^\infty$.
The notion of $\bt$-regularity \cite{ST,Tib99} is 
about the nature of the relative conormal space at the divisor at infinity
of a projective compactification of the graph of a complex or real polynomial 
function. In our present and most general real and definable setting, 
the analogue of the divisor at infinity would be the "boundary subset" 
$Z^\infty$. Using the Rabier number, it is much clearer and easier to present 
our avatar of the notion of $\bt$-regularity in terms of limits of tangent 
spaces (see Definition \ref{def:regular-infinity}). We introduce the 
following mapping
$$
\beta_s := \beta\times {\rm id}_{\Rs} : \bM\times\R^s \to \R^n\times\Rs.
$$
For any limit of tangent spaces $T$ of $(\cT_{r_\bz}^\infty)_\bz$, 
with $\bz = (0,\bu,\bc)$, the projection of 
$T$ onto $\Rs$ means the restriction to $T$ of the projection of
$T_\bz\bMbar_s = T_{(0,\bu)}\bMbar\times T_\bc\Rs$ onto the the second
factor $T_\bc\Rs$ of this product.

\begin{definition}\label{def:regular-infinity}
Let $W$ be a closed connected definable $C^1$ sub-manifold of 
$\Rn\times \Rs$ of dimension $d\geq s$. Let $Z:=\beta_s^{-1}(W)$.

The subset $W$ is \em $\bt$-regular at the value $\bc$ \em  if
for any $\bz\in Z^\infty\cap \bMbar\times \bc$ and any 
$T\in \cT_{\br_Z,\bz}^\infty$ the projection mapping $T \to \Rs$ is 
surjective.

If $W$ is the graph of $F:X \to\Rs$, the mapping $F$ is \em $\bt$-regular
 at the value $\bc$ \em  if $W$ is.
\end{definition}
When working with a polynomial or semi-algebraic $C^1$ mapping $\Rn\mapsto\Rs$,
observe that \cite[Definition 2.10]{Tib99} and \cite[Definition 2.5]{DRT}
are somehow equivalent to our Definition \ref{def:regular-infinity}.

\bigskip
Equipping $[0,+\infty[$ and $\bS^{n-1}$ with their respective 
Euclidean metric, the space $\bMbar$ inherits this metric product structure,
and so does $\bMbar_s$.
\\
If $\bm =(\br,\bu,\bc)$ is any point of $\bMbar_s$ then $T_\bm\bMbar_s = 
\R \times T_\bu\bS^{n-1} \times T_\bc\Rs$. For any given subspace $T$  of 
$T_\bm\bMbar$, let $L:T\mapsto \Rs = T_\bc\Rs$ be the projection of
$T$ onto $\Rs$. The Rabier number $\nu_s(L)$ of $L$ is defined w.r.t the product 
metric on $\bMbar_s$ and the canonical one on $\Rs$.
Let $\pi_s(\bm)$ be the projection of $T_\bm\bMbar_s$ onto $T_{\bc}\Rs$.

\smallskip\noindent
We get the following useful and quantified characterization of $\bt$-regularity.
\begin{lemma}\label{lem:reg-inf-rabier}
The sub-manifold $W =\beta_s(Z)$
is $\bt$-regular at the regular value $\bc$ if and only if
there exists $\nu_\bc >0$ such for every $\bz\in Z^\infty\cap\bM^\infty\times\bc$
and every $T \in \cT_{r_Z,\bz}^\infty$ we have
$$
\nu_s\left(\pi_s(\bz)|_T\right)\geq \nu_\bc.
$$
%%where $Z = \beta_s^{-1}(W)$.
\end{lemma}
\begin{proof}
It is a straightforward consequence of the following ingredients: $Z^\infty 
\cap \bM^\infty \times \bc$ is definable and compact, the mapping $\bz \mapsto 
\nu_s(\pi_s(\bz)|_{T_\bz\br_Z})$ is definable and continuous, the projection 
of $(\dim W -1)$-planes of $\Rn\times\Rs$ onto $\Rs$ is continuous, and the
definition  of $\bt$-regularity.
\end{proof}

%
%
%
%
%
%
%
%
%
%
%
%
%
%
%
%  
% 
%*******************************************************************
%
%
%
%
%
%
%
%
%
%
%
%
%
%
%
% 
\section{Rabier number, Malgrange-Rabier condition}\label{section:rabier}

Let $\bw := (\bx,\by)$ be the Euclidean coordinates of $\R^n\times\R^s$.
Let $W$ be a closed connected definable $C^2$ definable sub-manifold of 
$\Rn\times \Rs$ of dimension $d\geq s$ and let $\vp:W \to \Rs$ be the
projection $(\bx,\by) \mapsto \by$. We 
introduce the following (see \cite{Rab,KOS}):
\begin{definition}\label{def:Malgrange-Rabier}
The mapping $\vp$ satisfies the \em Malgrange-Rabier condition $(MR)$ at 
the value $\bc\in\Rs$ \em if there exists a positive constant $L$ such that
\begin{equation}\label{eq:malgrange-rabier}
|\vp(\bw) - \bc| \, \ll \, 1 \; {\rm and} \; |\bw| \, \gg \,1\,  
\Longrightarrow |\bw| \cdot \nu(D_{\bw}\vp) \geq L.
\end{equation}
We will also say that $\vp$ is \em (MR)-regular at $\bc$. \em
\end{definition}
In Definition \ref{def:Malgrange-Rabier}, the mapping $\vp$ tends to
a finite value $\bc$, therefore the asymptotic phenomenon at 
infinity only occurs in the directions of the $\Rn$ component.
In other words,
since $\bw = (\bx,\by) \in \Rn\times\Rs$, 
instead of requiring $|\bw| \gg 1$ and $|\bw| \cdot |D_\bw \vp| \geq L$,
we can equivalently require $|\bx| \gg1$ and $|\bx| \cdot |D_\bw \vp| \geq L$.

As counterpart to (MR)-regular values we explicitly introduce the 
next
\begin{definition}\label{def:ACV}
An \em asymptotic critical value $\bc$ of the mapping $\vp$ is simply \em a 
Malgrange-Rabier critical value of $\vp$, namely 
there exists a sequence $\bw_k := 
(\bx_k,\by_k)_k$ of points of $W$ such that: (i) $\bw_k \to \infty$, (ii) 
$\by_k \to \bc$ and
$$
|\bw_k| \cdot \nu(D_{\bw_k}\vp) \to 0.
$$
\noindent
It is customary to write $K_\infty(\vp)$ for the set of ACVs of the mapping 
$\vp$.
\end{definition}
\smallskip
\begin{definition}\label{def:GCV}
A \em generalized critical value of $\vp$ \em is any  value of the subset
$$
K(\vp) := K_0(\vp) \cup K_\infty (\vp) \, \subset \, \R^s.
$$
\end{definition}

\medskip
\noindent
{\bf Notation:} \em Let $\nu^\vp(\bw) := \nu(D_{\bw}\vp)$.\em

\medskip
We start with the following expected result:
\begin{lemma}\label{lem:Kf-def-closed}
The subset $K(\vp)$ is definable and closed.
\end{lemma}
\begin{proof}
Since the Malgrange-Rabier condition proceeds from a first order formula in 
the language of the structure $\cM$, the subset $K_\infty(\vp)$ is definable 
in the structure. 
Therefore the subset $K(\vp)$ is too. 
\\
Compactify $\R^q$ as $\bS^q$, and $\R_{\geq 0}$ as $[0,+\infty]$, by adding 
a single point at infinity. The compactification is semi-algebraic. Consider
the following compact definable subset
$$
G := \clos \{(\bw,(1+|\bw|)\cdot\nu^\vp(\bw)) \in W \times \R_{\geq 0} \} 
\subset \bS^n \times\bS^s \times [0,+\infty].
$$
Denoting by $g:G\to \bS^s$ the projection on $\bS^s =\Rs \sqcup\infty$, 
the following identity
$$
K(\vp) = \Rs \cap g \left(G\cap \bS^n\times \bS^s\times 0\right)
$$
concludes the proof.
\end{proof}
%
%
%%Theorem \ref{thm:morse-sard} below is a
%%(somehow new) proof of it being of positive codimension.

%%\smallskip
%%Let $\cL(W,s) := {\rm Hom}_\R(TW,W\times\Rs)$ be the (definable) $C^1$ 
%%vector bundle over $W$ with fibre $\cL(T_\bw W,\Rs)$.
The next result tells us about the 
behaviour at infinity nearby a value $\bc \notin K(\vp)$ of some feature of 
the mapping $\vp$. 
\begin{lemma}\label{lem:rabier-spherical}
Let $\bc$ be a regular value of $\vp$ which is also (MR)-regular. 
Let $(\bw_k)_k$ be any sequence of $W$ such that (i) $\bw_k \to \infty$, 
(ii) $\by_k \to \bc$. The following limit holds true
\begin{equation}\label{eq:rabier-spheric}
\lim_{k \to +\infty} \, \dlt_{n+s} \left(\frac{\bw_k}{|\bw_k|} , 
T_{\bw_k} \vp\right) = 0.
\end{equation}
%
%
%so that, $\bu \in T$.
\end{lemma}
\begin{proof}
By hypothesis, there exists a positive constant $L$
such that along any sequence $(\bw_k)_k$ as in the statement we
have
$$
|\bw_k| \cdot \nu^\vp(\bw_k) \, \geq \, L, \; {\rm for} \; k \gg 1.
$$

It is enough to work with paths definable at $+\infty$ instead of sequences.
Let
$$
t \mapsto \bw(t) = (\bx(t),\by(t))\in W
$$ 
be a germ of definable mapping at $+\infty$, such that as $t$ goes to $+\infty$ 
$$
\bw(t) := t \bu(t), \; {\rm with} \, \bu(t)\in \bS^{n+s-1}; \; 
\by(t) \, \to \, \bc.
$$ 
In particular 
$\bu(t) \to \bu\in\bS^{n-1}\times \bbo$.
Since $\vp$ is definable, the tangent-to-the-fibre mapping 
$$
\bw \, \mapsto \, T_\bw \vp 
\in \bG(\dim W-s,n+s)
$$
is definable. Thus the following  mapping 
is definable at $+\infty$
$$
t\mapsto T_{\bw(t)} \vp, \;\; {\rm and} \;\; T_{\bw(t)}\vp \to T
$$
as $t$ goes to $+\infty$. We want to show that 
\begin{center}
$
\lim_\infty \dlt_{n+s}(\bw',T_\bw \vp) = 0,
$
\end{center}
since Lemma \ref{lem:secant-tangent} gives that 
$\left|\bw' - \frac{\bw}{|\bw|}\right| \to 0$ at infinity. 
The mapping
$$
t \mapsto f(t) := \vp(\bw(t)) - \bc
$$
is definable at $+\infty$ and tends to $\bbo$ as $t$ goes to $+\infty$. 
Working with each component of $f$, Lemma \ref{lem:derivative-infinity} yields 
$$
t\cdot f'(t) \to \bbo.
$$
Let $N_\bw \vp \subset T_\bw W$ be the orthogonal complement of $T_\bw \vp$.
Let $\bw' = \bw_T' + \bw_N'$ be the decomposition of $\bw'$ in the orthogonal 
direct sum $T_{\bw(t)} \vp \oplus N_{\bw(t)} \vp = T_{\bw(t)} W$.
Of course the germ of each of the following mappings 
$$
t\mapsto T_{\bw(t)} W, \;\; t \mapsto N_{\bw(t)} \vp, \;\;
t \mapsto \bw_T',\,\; {\rm and} \;\; t\mapsto \bw_N'
$$ 
is definable at $+\infty$. By definition we have
$$
\dlt_{n+s}(\bw',T_\bw \vp) = |\bw_N'|.
$$
Combining the previous equality with the hypothesis, we find 
$$
|D_\bw \vp \cdot \bw'| =  |D_\bw \vp \cdot \bw_N'|\geq | \bw_N'|\cdot 
\nu^\vp(\bw(t)),
$$
from which we get
$$
|\bw_N'(t)| \leq \frac{t\cdot |f'(t)|}{L} \to 0.
$$
\end{proof}
We introduce the 
following:
\begin{definition}\label{def:increase-fast}
A germ of a continuous, positive function $\mu : (\Rgo,+\infty) \to \Rgo$
is \em increasing fast  at $+\infty$ \em if
$$
\lim_{t\to+\infty} t^{-1}\cdot\mu(t) = +\infty.
$$
For a function germ $\mu$ increasing fast at $+\infty$, we define
\begin{center}
$
K_\infty^\mu(\vp) = \{\bc\in\Rs \, : \, \exists (\bw_k)_k\in W \; {\rm with} 
\; \bw_k \to +\infty, \, \by_k \to \bc \; {\rm such \; that} \; 
\mu(|\bw_k|)\cdot\nu^\vp(\bw_k) \to 0\}.
$
\end{center}
\end{definition}

\smallskip
In Definition \ref{def:increase-fast} we can, equivalently,
require that $\mu(|\bx_k|) \cdot \nu^\vp(\bw_k) \, 
\to \, \infty$. 
When the function germ $\mu$ is definable at $+\infty$, the subset 
$K_\infty^\mu(\vp)$ is definable, contained in $K_\infty(\vp)$.
As in \cite{dAc,KOS}, we have
\begin{proposition}\label{prop:K-inf-p}
There exists a function germ at infinity $\mu:(\Rgo,+\infty) \to \Rgo$ 
increasing fast and definable at $+\infty$ such that
$$
K_\infty(\vp) = K_\infty^\mu(\vp).
$$
\end{proposition}
\begin{proof}
The principle of the proof is the same as that of \cite[Lemma 3.3]{dAc} or 
\cite[Lemma 3.1]{KOS}. Using the inverse of the stereographic projection, any 
Euclidean space $\R^q$ compactifies smoothly and semi-algebraically as 
$\bS^q~=~\R^q \sqcup \infty$.
Let $\clos(W)$ be the closure of $W$ in $\bS^{n+s}$.
We consider the following definable mapping
$$
G: \, W \; \to \Rs\times\R, \; \bw \, \mapsto \, (\vp(\bw),|\bw|
\cdot\nu^\vp(\bw)).
$$ 
Let $Y:=\clos(W) \times \bS^s \times \bS^1$. Define
$$
\Gamma := \clos({\rm graph}(G)) \subset Y
%% :=\clos(W) \times \bS^s \times \bS^1, 
\; {\rm and} \;
A := \infty\times \clos(K_\infty(\vp)) \times 0 \subset Y.
$$
Observe that $A$ is closed. The definable Wing Lemma (\cite{Whi}, 
\cite[Lemma 2.7]{Loi1}) states that there exists a 
closed definable subset $B$ of $\Gm$ such that 
$$
B\cap A = \emptyset, \;\; {\rm and} \;\; \clos(B)\cap \Gm\setminus {\rm
graph}(G) = A.
$$
The point $\infty\times \bu \times 0$ belongs to $A$ if there exists a sequence 
$(\bw_k)_k \in \pi_W(B)$ such that $\bw_k \to \infty$, $\by_k \to \bc$ and 
$|\bw|\cdot\nu^\vp(\bw) \to 0$, where $\pi_W :\Gamma \mapsto \clos(W)$ is the 
projection onto the first factor. Consider the following 
definable function
$$
 b(r) := r\cdot\max\{\nu^\vp(\bw)\; {\rm for} \; \bw\in\pi_W(B) \;
{\rm and} \; |\bw| = r\}.
$$
We check as in \cite[Lemme 3.3]{dAc} that the function defined as 
$$
r \mapsto \mu(r) := \frac{r}{\sqrt{b(r)}}
$$
satisfies the announced statement.
\end{proof}
To conclude this section, we would like to add a few words about the special 
case of graphs.

\smallskip
Let $X$ be a closed connected definable $C^2$-sub-manifold of $\R^n$. 
Let $F:X \to\Rs$ be a $C^2$ and definable mapping. 
We suppose that $\dim X \geq s$.
The corresponding Rabier number function to consider is
$$
\bx \mapsto \nu^F(\bx) := \nu(D_\bx F).
$$
We recall the following complementary definitions: 

\em A  value $\bc$ is \em a Malgrange-Rabier regular value of the mapping $F$ 
\em 
if there exists a positive constant $L$
such that for any sequence $(\bx_k)_k$ of $X$ such that
(i) $\bx_k\to+\infty$, (ii) $F(\bx_k)\to \bc$ , we have
$$
|\bx_k|\cdot \nu^F(\bx_k) \geq L.
$$

A value $\bc\in\Rs$ is \em an ACV of the mapping $F$ \em if there exists a 
sequence $(\bx_k)_k$ of $X$ such that
(i) $\bx_k\to+\infty$, (ii) $F(\bx_k)\to \bc$ and (iii)
$|\bx_k|\cdot \nu^F(\bx_k) \to 0$.  \em

\smallskip
Let $W$ be the graph of $F$, and let $\bw =(\bx,F(\bx))$ be a point of $W$.
Since $T_\bw W$ is just the graph of the linear mapping $D_\bx F:T_\bx X 
\to\Rs$, 
Equation \eqref{eq:nu-projection} of Lemma \ref{lem:nu-projection}
guarantees that the value $\bc\in\Rs$ lies in $K_\infty(F)$ if and 
only if it lies in $K_\infty(\vp)$.

\smallskip
Any statement and notion for $\vp$ presented in this section admits a version 
for the mapping $F$. The modifications to be done are clear, and we can check 
that each demonstration for a statement about $F$ adapts easily from the 
corresponding statement for $\vp$, following the same steps.
%
%
%
%
%
%
%
%
%
%
%
%
%
%
%
%
%
%
%
%
%    *******************************************************************
%
%
%
%
%
%
%
%
%
%
%
%
%
%
%
%
%
%
%
\section{Malgrange-Rabier condition and
geometry at infinity}\label{section:MR-treg}
We are going to combine here the local point of view at infinity of 
Section \ref{section:t-reg} and the affine one of Section \ref{section:rabier}.

\medskip
The blowing-up at infinity mapping 
$\beta:\bM \to\Rn\setminus\bbo$, although being a diffeomorphism is by 
no means an isometry. Nevertheless spheres $r\times \bS^{n-1}$ are mapped 
onto spheres $\bS_{\frac{1}{r}}^{n-1}$, angles of pairs of vectors tangent to 
sphere at a same point are preserved and orthogonality to the spheres as 
well, as we can check below:
Let $\bm= (\br,\bu)$ be a point of $\bM$. We find 
$$
D_\bm \beta|_{T_\bu \bS^{n-1}} = \br^{-1} {\rm Id_{T_\bu \bS^{n-1}}}, 
\; \mbox{ and } \;
D_\bm \beta \cdot \dd_\br = - r^{-2} \bu = -\br^{-1} \beta(\bm) .
$$

Working with $Z$ nearby $\br_Z = 0$ is working in a neighbourhood at infinity
of $W$. We find the following:
\begin{lemma}\label{lem:rabier-surjective}
If $\vp$ is (MR)-regular at the regular value
$\bc$, then $W$ is $\bt$-regular at $\bc$.
\end{lemma}
\begin{proof}
%%Following Lemma \ref{lem:reg-inf-rabier} 
It is sufficient to show the result along a path at infinity.
Let $t\mapsto \bw(t) \in W$ be the germ of a mapping definable at $+\infty$ 
such that:
$\bz(t) := \beta_s^{-1}(\bw(t)) \to \bz_\infty = (0,\bu,\bc) \in 
\bM_s^\infty$, for $\bc$ a (MR)-regular value of $\vp$, and
$$
T(t) := T_{\bz(t)} Z \to T, \;\; {\rm and}  \;\;
R(t) := T_{\bz(t)} \br_Z \to R.
$$
Without loss of generality we further assume that $\bw(t) = t\bu + o(t)$, 
and set
$$
\bu(t) = \frac{\bx(t)}{|\bx(t)|} = \frac{\bw(t)}{|\bw(t)|} + o(1).
$$
We recall that 
$$
R(t) \subset T(t) \cap 0 \times T_{\bu(t)}\bS^{n-1}
\times \Rs.
$$
Let $D(t) := D_{\bw(t)} \vp$.
Let $\nu(t) = \nu^\vp(\bw(t))$. By hypothesis we have
$$
t\cdot\nu(t)\geq L, \;  \mbox{for some constant}\; L >0.
$$
Let us still denote by $D(t)$ the extension of $D(t)$ to $\Rn\times\Rs$ as 
the null mapping over $(T_{\bw(t)}W)^\perp$. 
In other words, 
%%there exists $\eta > 0$ such that \tcrd{We can take $2\eta = 1$ -
%%change accordingly.
we have
$$
t \gg 1 \Longrightarrow \frac{1}{\nu(t)}D(t) \cdot \bB^{n+s}(\bbo,1) 
\supset \bB^s(\bbo,1).
$$
Let $\xi_s$ be any vector of $\bS^{s-1}$. There exists a germ of mapping 
definable at infinity  
$$
t\,\mapsto \, \chi(t) \, \in \, \bB^{n+s}(\bbo,1)
$$ 
tangent to $X$ along $\bw$, such that for $t$ large enough
$$
\xi_s = \frac{1}{\nu(t)} D(t) \cdot \chi(t) \in \bS^{s-1}.
$$
In the orthogonal direct sum
$$
T_{\bw(t)} (\Rn\times\Rs) = \R\bu(t)\oplus T_{\bu(t)} \bS^{n-1}\oplus\Rs,
$$
we decompose $\chi(t)$ as the definable orthogonal sum
$$
\chi(t) = \psi(t)\bu(t) + \chi_\bS(t) + \nu(t)\xi_s.
$$
We find 
$$
\xi(t) :=D_{\bw(t)} \beta_s^{-1} \cdot \frac{1}{\nu(t)}\chi(t) = 
\left( \frac{-\psi(t)}{t^2\nu(t)},\frac{1}{t\nu(t)}\chi_\bS (t),\xi_s\right).
$$
Since $|\chi (t)| \leq 1$ and $t\cdot\nu(t) \geq L$ once $t$ is large 
enough, we deduce there exists $\xi_\bS \in T_\bu\bS^{n-1}$ such that
$$
\xi(t) \to (0,\xi_\bS,\xi_s) \in \R\times T_\bu \bS^{n-1} \times \Rs,
$$
proving the announced result. Indeed the vector $(0,\xi_\bS,\xi_s)$ lies
necessarily in $R$, since $R$ is a subspace of $T\cap 0\times 
T_\bu\bS^{n-1}\times\Rs$.
\end{proof}
The next lemma is a converse of Lemma \ref{lem:rabier-surjective}.
\begin{lemma}\label{prop:t-implies-RM}
If  $W$ is $\bt$-regular at the regular value $\bc$ of $\vp$, then $\vp$ is 
(MR)-regular at $\bc$.
\end{lemma}
\begin{proof}
Again, it is enough to check that the property holds true along a 
definable path at infinity.

Let $t\mapsto \bw(t) : = (\bx(t),\by(t))$ be a germ of a mapping definable 
at $+\infty$, and let $\bz = \beta_s^{-1}\circ\bw$. Let us consider the following 
limits as $t\to+\infty$
$$
\bz(t) \to \bz_\infty := (0,\bu,\bc) \in \bMbar_s^\infty, 
%%\;\; T(t):= T_{\bz(t)} Z \to T,
\;\; {\rm and} \;\; 
R(t) := T_{\bz(t)} \br_Z \to R,
$$
where $Z:=\beta_s^{-1}(W)$.
We can assume that $\bz(t) = (t^{-1},\bu(t),\by(t))$ with $\by(t)) 
\to \bc \notin K(\vp)$.
Let 
$$
\pi_s (t) := \pi_s(\bz(t)) \;\; {\rm and} \;\;
\nu_s(t) := \nu_s\left(\pi_s (t)|_{T_{\bz(t)}\br_Z}\right).
$$
The hypothesis implies that 
$$
\lim_{t\to +\infty}\nu_s(t) = \nu \geq \nu_\bc.
$$
Let $\xi$ be a unit vector of $\Rs$. Let $\xi(t)$ be a definable
path with values in $T\br_Z$, along the path $t \mapsto \bz(t)$ such that
for $t$ large enough we have 
$$
\pi_s(t)\cdot\xi(t) = \xi.
$$
Therefore we can write the following orthogonal decomposition
$$
\xi(t) = r(t)\dd_\br\oplus\zt(t)\oplus\xi\in  T_{\bz(t)}\bM_s = \R \oplus 
T_{\bu(t)} \bS^{n-1} \oplus \Rs,
$$
so that the hypotheses yield
$$
r(t) = 0, \;\; {\rm and} \;\; 
|\xi(t)| \leq \frac{1}{\nu_s(t)} \leq \frac{2}{\nu}.
$$
Let $t \mapsto \chi(t)$ be the definable path at infinity
$$
\chi(t) := D_{\bz(t)}\beta_s(t)\cdot\xi(t) \in T_{\bw(t)}W. 
$$
Decomposing $\chi(t)$ in the orthogonal direct sum $\R\bu(t)\oplus 
T_{\bu(t)}\bS^{n-1} \oplus \Rs$, we get
$$
\chi(t) = 0\oplus\bv(t)\oplus\xi.
$$
In particular we obviously find 
$$
D_{\bw(t)}\vp \cdot \chi(t) = \xi.
$$
By the choice of parametrization of $\bz$ we have $|\bx(t)| = t$. 
Since 
$$
\xi(t) = 0\oplus\frac{1}{t}\bv(t)\oplus\xi,
$$
we observe furthermore that
$$
\frac{2}{\nu}\cdot t \; \geq \; t\cdot |\xi(t)| \; \geq \; |\chi(t)|.
$$
Since $|\bw(t)|\geq t$, the previous inequalities imply the following inclusion
$$
\bB_1^s \subset D(t) \cdot \bB_{\frac{2}{\nu}|\bw(t)|}^{n+s},
$$
where $D(t)\in \cL(n+s,s)$ is the linear extension of $D_{\bw(t)} \vp$ to 
$\Rn\times\Rs$, as the null mapping over the normal space $N_{\bw(t)}W$
of $T_{\bw(t)}W$ in $\Rn\times\Rs$. In other 
words we have proved that
$$
\nu^\vp(\bw(t)) = \nu(D(t)) 
%%\geq \frac{\nu}{2} \cdot \frac{1}{t} 
\geq \frac{\nu_\bc}{2} \cdot \frac{1}{|\bw(t)|},
$$
which ends the proof.
\end{proof}
Combining Lemma \ref{lem:rabier-surjective}, 
Lemma \ref{prop:t-implies-RM} we have proved in our most general context the
following
\begin{proposition}\label{prop:T-equiv-MR}
A regular value $\bc$ of $\vp$ is (MR)-regular if and only if
the sub-manifold $W$ is $\bt$-regular at the value $\bc$.
\end{proposition}
In other words, we have rigorously showed, in this most general context, 
that (MR)-regularity and (the current avatar of) $\bt$-regularity are equivalent (see 
\cite{Tib99,DRT,DuGr1} for special cases).
\begin{remark}\label{rmk:d=s} When $d=s$, the condition of Malgrange-Rabier 
regularity at $\bc$ is equivalent to the local properness of the mapping $\vp$
at the value $\bc$, which means that $\vp$ is proper over some neighbourhood of 
$\bc$. Indeed the ACVs coincide with the non-properness values of the mapping  
$\vp$ (see \cite[pages 86 to 88]{KOS} and adapt almost readily their 
arguments to our context). In particular Lemma \ref{lem:rabier-surjective} 
and Lemma \ref{prop:t-implies-RM} are trivial.
\end{remark}
Using Proposition \ref{prop:K-inf-p} we are now in position to show the
following Morse-Sard  type result, whose proof is different from 
those given in \cite{JeKu,KOS} and gives a more precise insight of the 
geometric phenomenon at hand than that of \cite{DRT}.
\begin{theorem}\label{thm:morse-sard}
The subset $K_\infty (\vp)$ of $\R^s$ (is definable) and has positive 
codimension.
\end{theorem}
\begin{proof}
By Proposition \ref{prop:K-inf-p} we know that $K_\infty (\vp) = 
K_\infty^\mu(\vp)$, for some fast decreasing function $\mu$ definable
at $+\infty$.
The following subset
$$
W_1 := \{\bw\,\in\,W\,:\, \mu(|\bw|)\cdot \nu^\vp(\bw) \leq 1\}
$$
 is a closed and definable subset of $\Rn\times\Rs$.
Let
$$
Z_1 := \beta_s^{-1}(W_1), \;\; {\rm and} \;\;
Z_1^\infty := \clos(Z_1) \cap \bM_s^\infty.
$$
By Lemma \ref{lem:rabier-surjective} and Lemma \ref{prop:t-implies-RM}, 
any point $\bz_\infty =(0,\bu,\bc)$ of $Z_1^\infty$ is a limit of a 
sequence $(\bz_k)_k$ of $Z_1$ such that $T_{\bz_k}\br_Z \to R\subset
T_{\bz_\infty} \Mbar_s$ which does not projects surjectively onto 
$\Rs$.
Let $\tau_Z$ be the projection of $\clos(Z)$ onto $\Rs$.
Thus 
$$
\tau_Z(Z_1^\infty) \, = \, K_\infty (\vp).
$$
Let $\br_1$ be the restriction of $\br_Z$ to $\clos(Z_1)$. It is a 
continuous definable mapping. The pair $(Z_1,Z_1^\infty)$ of germs at 
$\bM_s^\infty$ can be definably stratified to be $(\ba_{rel})$-regular
w.r.t. the function $\br_1$ \cite{Loi3}.
Let $\bz_\infty$ be a point of $S^\infty$, a strata contained in $Z_1^\infty$. 
Therefore $T_{\bz_\infty} S^\infty$ is contained in the limit 
$T$ at $\bz_\infty$ of the relative tangent spaces $T_{\bz_k}\br_Z = 
T_{\bz_k}\br_1$, given any sequence $(\bz_k)_k$ of $Z_1$ converging to 
$\bz_\infty$. By Proposition \ref{prop:T-equiv-MR} and the definition of $Z_1$,
the limit $T$ cannot project surjectively onto $\Rs$. Therefore we deduce 
$$
{\rm rank} \, D_{z_\infty}\left(\tau_Z|_{S^\infty}\right) \leq s-1,
$$
and thus we find
$$
\dim \tau_Z(Z_1^\infty) \leq s-1.
$$
%%By Proposition \ref{prop:K-inf-p} the result is proved.
\end{proof}
%
%

%
%
%
%
%
%
%
%
%
%
%
%
%
%
%
%
%
%
%
%
%
%                          *******************************************************************
%
%
%
%
%
%
%
%
%
%
%
%
%
%
%
\section{Bifurcation values and triviality results}\label{section:bif}

We recall here in this most general setting the notions and 
results about equi-singularity nearby a (MR)-regular value $\bc$ in 
the context of the previous sections.

\bigskip
Let $X$ be a connected, definable $C^a$ sub-manifold of $\Rn$ which is 
also a
closed subset of $\Rn$, where $a\in \N_{\geq 2} \cup \infty$. 
Let $F:X \to \R^s$ be a $C^k$ definable mapping where
$2\leq k \leq a$. In this definition $C^k$ is understood as the maximal 
possible regularity of $F$ on $X$. In order to have an interesting statement
we assume that $\dim X \geq s$.
\begin{definition}\label{def:typical}
A value $\bc$ of $\Rs$ is  \em typical for the mapping $F$ \em if
there exists a neighbourhood $\cV$ of $\bc$ in $\Rs$ such that
the restriction
$$
F|_\cU: \cU:=F^{-1}(\cV) \to \cV
$$
is a trivial $C^{k-1}$-fibre bundle over $\cV$ with model fibre $F^{-1}(\bc)$ 
and  $F|_\cU$ is the projection mapping onto the base.
\end{definition}

\bigskip
We can now see the relation of (MR)-regularity condition previously 
introduced with equi-singularity in the following expected
\begin{theorem}\label{thm:trivialisation}
If a regular value $\bc$ of $\Rs$ is (MR)-regular for the mapping $F$, then 
$\bc$ is a 
typical value for $F$.   
\end{theorem}
There are many occurrences of this result over the last forty years, 
mostly for functions though \cite{HaLe,ST,Par,Rab,Tib99,TLLZa,KOS,Jel1,Jel2,
dAc,JeKu,DRT,DuGr1} (this list is far from being exhaustive). 
The proof we present here is explicit and does not get into the messy task 
of composing $s$ flows (which is rarely explicitly done once $s\geq 2$). We 
follow the simple and beautiful idea of Rabier \cite{Rab} (see 
also Jelonek's variation in \cite{Jel1,Jel2} using Gaffney's characterization
of (MR)-regularity \cite{Gaf}).
\begin{proof}
Assume that $\bc$ is a (MR)-regular value of $F$. 
%%For any $\bx$ of $W$ let 
%%$$
%%\nu^F(\bx) := \nu(D_\bx F).
%%$$
There exists a positive number $\ve$ such that the closed ball
$\bB := \bB^p(\bc,\ve)$ does not intersect with $K(F)$.
Let $\bF$ be the closed definable $C^k$ sub-manifold with boundary 
$$
\bF := F^{-1}(\bB) \subset X.
$$
Observe that a sequence $(\bx_k)_k$ of $\bF$ escaping $\bF$ as $k$ goes to 
infinity either goes to the boundary $\dd \bF$ or goes to infinity (see
Condition (4.1) in  \cite[Lemma 4.1]{Rab}).
Up to a translation we can assume that the origin of $\Rn$ does 
not lie in $\bF$. 
By definition of Malgrange-Rabier regularity, we can assume that 
there exists a positive constant $M$ such that the following estimates holds
true in $\bF$:
\begin{equation}\label{eq:MR-global}
|\bx| \cdot \nu(D_\bx F) \geq M.
\end{equation}
For $\bx\in \bF$, let $A(\bx) := D_\bx F$. We consider the following 
linear mapping
$$
\bV(\bx) := A(\bx)^t \cdot (A(\bx)\cdot A(\bx)^t)^{-1}.
$$
Following Rabier's terminology, $\bV(\bx)$ is a right inverse of $A(\bx)$:
$$
A(\bx)\cdot\bV(\bx) = {\rm Id}_{T_{F(\bx)}}\Rs.
$$
The mapping  $\bx \mapsto \bV(\bx)$ is a $C^{k-1}$ definable section of 
the vector bundle $F^*(T\Rs)|_\bF$ and takes values in 
$NF|_\bF \subset T\Rn|_\bF$, the normal bundle of $F$ over $\bF$.
We follow now Rabier's proof \cite[Section 4]{Rab}. Consider
the following differential equation with parameters:
\begin{equation}\label{eq:ODE-rabier}
\left\{
\begin{array}{ccl}
\dot{\bx} & = & \bV(\bx)\cdot(\by-\bc) \\
\bx(0) & \in & \bF
\end{array}
\right.
\end{equation}
where $\by \in \bB$. 
The vector field with parameters
$$
\xi(\bx,\by) := \bV(\bx)\cdot(\by-\bc)
$$
is definable and $C^{k-1}$.
If $\Psi$ denotes the flow of ODE \eqref{eq:ODE-rabier}, observe that
$$
F(\Psi(t,\bx,\by)) = F(\bx) + t(\by-\bc).
$$
Let $\bz = (\bx,\by)$ be in $\bF\times\bB$ and  let $I_\bz$
be the domain over which the solution 
$$
t \mapsto  \vp (t;\bz)
$$
of the ODE \eqref{eq:ODE-rabier} exists and where $\vp(0;\bz) = \bx$.
Let
$$
\ell(t;\bz) := \left|\int_0^t |\bV(\vp(\tau;\bz))|\rd\tau\right|,
$$
be the length of the trajectory $t \mapsto \vp(t;\bz)$ between the time $0$ 
and $t$. Using Estimate \eqref{eq:MR-global} combined with a Gr\"onwall 
argument (see \cite{dAGr1}), we find that
$$
\ell(t;\bz) \leq \frac{1}{M}\cdot|\by-\bc|\cdot|\bx|\cdot e^\frac{t}{M}.
$$ 
From here we conclude as in the proof of \cite[Lemma 4.2]{Rab}, that for
each $\bx\in F^{-1}(\bc)$ the interval $I_\bz$ contains $[-1,1]$.
The trivialization is achieved as the $C^{k-1}$ diffeomorphism 
\begin{equation}\label{eq:MR-flow}
F^{-1}(\bc)\times \bB \to \bF, \;\;  
(\bx,\by) \mapsto \Psi(1,\bx,\by).
\end{equation}
\end{proof}
Let $\cV$ be any connected component of $\Rs\setminus K(F)$. Let 
$\cU_1,\ldots,\cU_\aph$ be the connected components of $F^{-1}(\cV)$ and 
$F^a$ be the restriction of $F$ to $\cU_a$. Theorem \ref{thm:trivialisation} 
implies the following:
\begin{corollary}\label{cor:trivialisation}.   
For each $a=1,\ldots,\aph$, the mapping $F^a: \cU_a \mapsto \cV$ induces a 
locally trivial $C^{k-1}$ fibre bundle structure over $\cV$, with connected 
model fibre $F^{-1}(\bc_a) \cap \cU_a$ for some $\bc_a \in \cV$.
\end{corollary}
Of interest is the following
\begin{remark}\label{rmk:flow-analytic}
Assume that $F$ is real analytic and globally sub-analytic.
The proof of Theorem 
\ref{thm:trivialisation} shows that nearby (MR)-regular values the produced
flow of Equation \eqref{eq:MR-flow} is real analytic, since $\xi$ is a real 
analytic mapping.
\end{remark}
\begin{definition}\label{def:bifurcation}
A value $\bc$ of $\Rs$ is a \em bifurcation value of the mapping $F$ \em
if it is not typical. \em
%%We will also speak of an \em atypical value of $F$. 
Let $Bif(F)$ be the set of bifurcation values of $F$. 
\end{definition}
The (expected) corollary of Theorem \ref{thm:trivialisation} and Definition 
\ref{def:bifurcation} is the following
\begin{corollary}\label{cor:bif-kf}
$Bif(F) \subset  K(F)$.
\end{corollary}
%
%
%
%
%
%
%
%
%
%
%
%
%
%
%
%
%
%
%
%
%
%
%
%
%                          
%      ***********************************************************
%
%
%
%
%
%
%
%
%
%
%
%
%
%
%
\section{Sub-level sets family associated with a family of hypersurfaces}\label{section:sub-levels}
This is the setting that will be used in Section \ref{section:cont-curv-hyper}.

\medskip
Let  $F : \Rn \times \Rs \to \R$ be a $C^2$ definable function such that 
$0$ is a regular value of $F$. Define the following $C^2$ closed definable 
subsets of $\Rn \times \Rs$,
$$
\cW := \{\bp \in \Rn\times\Rs : F(\bp) \leq 0\}, \;\; {\rm and} \;\; 
W :=  \{\bp \in \Rn\times\Rs : F(\bp) = 0\}.
$$ 
We assume that $\cW$ and $W$ are both not empty. Observe that $\cW$ 
is a definable $C^2$ sub-manifold with definable $C^2$ boundary $W$. 
Let $\vp: W \to \Rs$, $\omg: \cW \to \Rs$ be respectively the 
projection mappings $(\bx,\by) \to\by$.
For each parameter $\by\in \Rs$, we define the function 
$$
f_\by :\Rn\to\R, \;\; {\rm as} \;\;\bx \mapsto F(\bx,\by),
$$
yielding the definable family of $C^2$ functions $(f_\by)_{\by\in\Rs}$.
For $\by \in \Rs$ we have
$$
\vp^{-1}(\by) = f_\by^{-1}(0) \times \by \subset \Rn\times\Rs
\;\; {\rm and}  \;\; 
\omg^{-1}(\by) = \{f_\by \leq 0\} \times \by  \subset \Rn\times\Rs.
$$
Let us  write 
$$
\cW_\by := \{\bx\in\Rn : f_\by(\bx)\leq 0\}, \;\; {\rm and } \;\;
W_\by:=  \{\bx \in \Rn : f_\by(\bx) = 0\}.
$$
Whenever $\by$ is a regular value of $\vp$, the level $W_\by$ 
is a $C^2$ closed definable hypersurface of $\Rn$ bounding the $C^2$ closed
definable sub-manifold with boundary $\cW_\by$.

We need some equi-singularity results about the family $(\cW_\by)_{\by \notin 
K(\vp)}$ similar to those presented in Section \ref{section:bif}, but for the 
mapping $\omg$ instead of simply $\vp$. 
\begin{proposition}\label{prop:trivial-sub-level}
For every $\bc \notin K(\vp)$, there exists a closed ball $\bB$ of $\Rs
\setminus K(\vp)$ centred
at $\bc$ such that $\omg^{-1}(\bB)$ is $C^1$ trivial
fibre bundle over $\bB$ with model fibre $\cW_\bc \times\bc$.
\end{proposition}
\begin{proof}
Following \cite[Theorem 6.11]{Cos}, the hypersurface $W$ admits a 
definable neighbourhood in $\Rn\times\Rs$ which is $C^1$ definably
diffeomorphic to 
$$
NW(\rho):=\{(\bw,\xi) \in NW : |\xi|\leq \rho(\bw)\},
$$
where $\rho$ is $C^2$ positive definable function on $W$ and 
$NW$ is the normal bundle of $W$ in $T(\Rn\times\Rs)$.
The diffeomorphism is obtained by the restriction of "the end point mapping"
$$
E:T(\Rn\times\Rs) \to \Rn\times\Rs, \;\; (\bp,\xi) \mapsto \bp + \xi
$$
to $NW(\rho)$. The aforementioned neighbourhood of $W$ is defined as
$$
\scrT_\rho := E(NW(\rho)).
$$
Let $N\vp$ be the normal bundle of $T\vp$ taken into $T(W \setminus 
\crit(\vp))$. It is $C^1$, definable and of rank $s$. Let $N_\rho\vp$ be 
the definable $C^1$ vector 
bundle of $T(\Rn\times\Rs)|_{\scrT_\rho}$ of rank $s$ obtained by parallel 
transport of $N\vp$: any $\bp\in \scrT_\rho$ writes uniquely as 
$\bp = \bw + \xi$ with $(\bw,\xi)$ in $NW(\rho)$, yielding,
when $\bw$ is not a critical point of $\vp$, 
$$
(N_\rho\vp)_\bp := N_\bw \vp.
$$
\noindent
Assume that $\bB = \bB^s(\bc,\ve)$ for some positive radius $\ve$.
For $\bp = (\bx,\by)\in \omg^{-1}(\bB)$, we define 
$$
A(\bp) := D_\bp \omg, \;\; {\rm and} \;\; V(\bp) := A(\bp)^t\cdot 
(A(\bp)\cdot A(\bp)^t)^{-1}.
$$
Observe that $V(\bp) : T_\by\Rs \mapsto \bbo\times T_\by\Rs \subset T_\bp \cW$ 
is the "identity" mapping of $T_\by\Rs$. 
Similarly we define for $\bp\in \scrT_\rho\cap\omg^{-1}(\bB)$ the following
linear operators:
$$
A_\rho(\bp) := D_\bp \omg|_{N_\rho \vp}, \;\; {\rm and} \;\; V_\rho (\bp) := 
A_\rho(\bp)^t\cdot (A_\rho(\bp)\cdot A_\rho(\bp)^t)^{-1}.
$$
For any $\bp$ in $\scrT_\rho$ we have $A_\rho(\bp) \cdot V_\rho (\bp) = 
{\rm Id}_{T_\by\Rs}$.
Note that the mapping $\bp \to V_\rho(\bp)$ is $C^1$ and definable.
Since any $\bp\in \scrT_\rho$ writes uniquely as $\bp = \bw + \xi$ we find
$$
A_\rho(\bp) = D_\bw\vp|_{N_\bw \vp}.
$$
There exist $C^1$ definable functions $a_0,b_0:\R_{\geq 0} \mapsto [0,1]$
such that: (i) $a_0 + b_0 \equiv 1$; (ii) $a_0^{-1}(0) = [1,+\infty[$; 
(iii) $b_0^{-1}(0) = [0,\frac{1}{2}]$. Any $\bp \in \scrT_\rho$ has a unique 
form $\bp = \bw + \xi$. Define the following mapping over 
$\omg^{-1}(\bB)$:
$$
\bp \mapsto \bV(\bp) :=
\left\{
\begin{array}{rcl}
V(\bp) & {\rm if} & \bt \notin \scrT_\rho \\
V_\rho(\bp) & {\rm if} & \bt \in \scrT_\frac{\rho}{2} \\
a_0\left(\frac{|\xi|}{\rho(\bp)}\right)V_\rho(\bp) + 
b_0\left(\frac{|\xi|}{\rho(\bp)}\right)V(\bp) & {\rm if} & \bt\in 
\scrT_\rho \setminus \scrT_\frac{\rho}{2} .
\end{array}
\right.
$$
It is definable and $C^1$ and satisfies 
$$
D_\bp \omg\cdot\bV(\bp) = {\rm Id}_{T_\by\Rs}.
$$
Observe that whenever $\bp$ lies in $\cW\setminus W$, we have 
$\nu(D_\bp\omg) =1$. By definition of $\bc$ and $\bB$ there exists a positive 
constant $M$ (depending on $\bB$) such that
$$
\bw \in \vp^{-1}(\bB) \Longrightarrow (1+|\bw|)\cdot\nu^\vp(\bw) \geq M.
$$
Observe that for any $\bp =\bw+\xi \in \scrT_\rho \cap \omg^{-1}(\bB)$, 
we have 
$$
\nu(A_\rho(\bp)) = \nu^\vp(\bw).
$$
We consider again the ODE  \eqref{eq:ODE-rabier}. The remarks about the
estimates of $\nu(D_\bp \omg)$ and $\nu^\vp(\bw)$ are just to guarantee 
that Gr\"onwall arguments of the proof of Theorem \ref{thm:trivialisation}
will go through so that the rest of this proof adapts readily to
our current $\bV$, and produces a $C^1$ flow that gives the announced 
trivialisation.
\end{proof}
Let $\cV$ be any  connected component of $\Rs\setminus K(\vp)$. Let 
$\cU_1,\ldots,\cU_\aph$ be the connected components of $\omg^{-1}(\cV)$ and 
$\omg^a$ be the restriction of $\omg$ to $\cU_a$. Proposition  
\ref{prop:trivial-sub-level} implies the following 
\begin{corollary}\label{cor:trivial-sub-level} 
For each $a=1,\ldots,\aph$, the mapping $\omg^a: \cU_a \to \cV$ induces a 
locally trivial fibre bundle structure over $\cV$, with connected model 
fibre $\omg^{-1}(\bc_a) \cap \cU_a$ for some $\bc_a \in \cV$.
\end{corollary}
%
%
%%In this setting we obviously also get a weaker variation of an analogue of 
%%Remark \ref{rmk:flow-analytic}.
%%%
%%%
%%\begin{remark}\label{rmk:sub-level-flow-analytic}
%%When $F$ is real analytic and globally sub-analytic, the sub-manifold with 
%%boundary $(\cW,W)$ is real analytic and globally sub-analytic.
%%The proof of Theorem \ref{thm:trivialisation} shows that nearby (MR)-regular 
%%values the produced flow of Equation \eqref{eq:MR-flow} is only  $C^\infty$
%%since we use a partition of unity type construction to exhibit a $C^\infty$ 
%%vector field (with analytic a parameters) to produce the trivialisation.
%%\end{remark}
%
%
%
%
% 
%
%
%
%
%
%
%
%
%
%
%
%
%
%
%
%
%
%
%
%
%
%
%
%
%
%                          
%      ***********************************************************
%
%
%
%
%
%
%
%
%
%
%
%
%
%
%
\section{On the link at infinity}\label{section:link-infty}

Let $S$ be a closed definable subset of $\R^p$. For any positive $R$, 
let
$$
S_R := S\cap\bS_R^{p-1}.
$$
The family $(S_R)_{R>0}$ is a definable family, thus 
there exists a positive radius $r_S$ such that the topological type of $S_R$ is 
constant once $R > r_S$ (see for instance \cite[Theorem 5.22]{Cos}).
In order to use the techniques developed in 
\cite{DutertreAdvGeo,DutertreGeoDedicata2012}, we recall the following key notion:
\begin{definition}\label{def:link-infty}
The \em link at infinity $\lk^\infty(S)$ of $S$ \em is any subset $S_R$ for 
$R> r_S$. 
\end{definition}

Let $W$ be a $C^2$ connected closed definable sub-manifold of $\Rn\times\Rs$
and let $\vp :W\to \Rs$ be the restriction to $W$ of the canonical 
projection $\Rn\times\Rs \to\Rs$. For any value $\by$ of $\Rs$, we 
recall that $\vp^{-1} (\by) = W_\by \times \by$.

Let $\cV$ be any  connected component of $\Rs\setminus K(\vp)$. Let 
$\cU_1,\ldots,\cU_\aph$ be the connected components of the germ at infinity 
of $\vp^{-1}(\cV)$.
Let us write $(\vp|_{\cU^a})^{-1}(\by) = W_\by^a \times \by$.
\\

The following result is probably known folklore, yet we will provide a proof.
\begin{proposition}\label{prop:link-infty-const}
Let $\bc$ be a regular and (MR)-regular value of $\vp$ lying in $\cV$. For each
$a = 1,\ldots,\aph$, there exists a closed ball $\bB$ of $\Rs$ centred at 
$\bc$ and not intersecting with $K(\vp)$ and there exists $R_\bB > 0$ such
that for each $\by \in \bB^\circ$ and each $R\geq R_\bB$ the links 
$(\vp|_{\cU^a})^{-1}(\by)_R$ and $(\vp|_{\cU^a})^{-1}(\bc)_{R_\bB}$ are $C^1$ 
diffeomorphic, and the links 
$(W_\by^a)_R$ and $(W_\bc^a)_{R_\bB}$ are $C^1$ diffeomorphic as well. 
\end{proposition}
\begin{proof}
We can assume that $\aph=1$.
Let $\psi : W\to \R_{\geq 0}\times\Rs$ be defined as $\bw \mapsto (|\bx|^2,\by)$
when $\bw =(\bx,\by)$. Lemma \ref{lem:rabier-spherical} implies that there 
exist a closed ball $\bB$ centred at $\bc$ and a positive radius 
$R_1$ such that $\psi$ is a proper submersion over $\vp^{-1}(\bB)\cap 
\{(\bx,\by) : |\bx|\geq R_1\}$.
Observe that the following closed definable subset 
$$
\bF := \vp^{-1}(\bB)\cap \{(\bx,\by) : |\bx|> R_1\}.
$$
is $C^2$ sub-manifold with boundary. The restriction of $\vp$ to the 
definable $C^2$ sub-manifold with boundary $(\bF,\dd \bF)$ is a 
$C^2$ submersion. 
%%, and the restriction of $\vp$ to $\dd \bF$ is a 
%%definable $C^2$ submersion over $\dd\bB$. It is an (annoying)
%%exercise to check that in this setting 
Ehresmann Fibration Theorem over $\bB^\circ$ yields the result once 
$R_\bB > R_1$. 

\smallskip
To get the result for the levels of $\vp$ instead of the family $(W_\by)_\by$, 
we follow the same scheme of proof as the previous one, only changing $\psi$ 
into the mapping $\bw \mapsto (|\bw|^2,\by)$.
\end{proof}
We complete this section dealing also with the links at infinity of the 
"family of the sub-level sets" of the hypersurface family context presented in 
Section \ref{section:sub-levels}. We keep the exact same notations.
\\
Let $\cV$ be any  connected component of $\Rs\setminus K(\vp)$. Let 
$\cU_1,\ldots,\cU_\aph$ be the connected components of the germ at infinity 
of $\omg^{-1}(\cV)$. Let us write 
$(\omg|_{\cU^a})^{-1}(\by) = \cW_\by^a \times \by$.
We have the following result. 
\begin{proposition}\label{prop:link-infty-const-sub-level}
Let $\bc$ be a regular and (MR)-regular value of $\cV$. 
For every $a=1,\ldots,\aph$, there exists a closed ball 
$\bB$ of $\Rs$ centred at $\bc$ and not intersecting with $K(\vp)$ 
and there exists $R_\bB > 0$ such that for each $\by \in \bB^\circ$ and 
each $R\geq R_\bB$ the links $(\omg|_{\cU^a})^{-1}(\by)_R$ and 
$(\omg|_{\cU^a})^{-1}(\bc)_{R_\bB}$ are $C^1$ diffeomorphic, and the links 
$(\cW_\by^a)_R$ and $(\cW_\bc^a)_{R_\bB}$ are $C^1$ diffeomorphic as well. 
%%
%%In other words, for every $\by \in \bB$ the link  at infinity 
%%$\lk^\infty(\cW_\by^a)$ (respectively $\lk^\infty((\omg|_{\cU^a})^{-1} (\by)$)
%%has $C^1$ constant type. 
\end{proposition}
\begin{proof}
We can assume that $\cV = \Rs\setminus K(\vp)$ and the germ at infinity
of $\omg^{-1}(\cV)$ are both connected.

\smallskip 
Consider the mapping $\psi: \Rn \times \Rs \to \R_{\geq 0}\times\Rs$ defined
as $(\bx,\by) \mapsto (|\bx|^2,\by)$.
For $\bc \in \Rs \setminus K (\vp)$, Proposition \ref{prop:link-infty-const} 
implies that there exist a small neighbourhood $U$ of $\bc$ in $\Rs$ and 
a compact subset $K_0$ of $\Rn$ such that 
$$
(\vp^{-1}(U) \setminus K_0\times U ) \cap \crit(\psi|_W) = \emptyset.
$$
Observe also the obvious fact that $\crit(\psi) = \bbo\times\Rs$.
Let $\bB$ be a closed ball of $\Rs$ centred at $\bc$ of radius $r>0$. 
The subset $K_1:= \bB_1^n\times\bB$ is compact in $\Rn \times \Rs$.

\smallskip\noindent
\begin{claim}\label{claim:K1}
$(\vp^{-1}(\bB) \setminus K_1 ) \cap \crit(\psi) = \emptyset$.
\end{claim}
\begin{proof}
If $(\bx,\by) \in \vp^{-1}(\bB)$ then $\by\in\bB$. If $(\bx,\by)$ lies  
in $\crit(\psi)$, then $\bx =\bbo$. 
\end{proof}
\noindent
Up to reducing $r$, we can assume that the closed ball $\bB$ is contained
in $U$. Let $R_1>0$ be such that $\bB_{R_1}^n$ contains $K_0\cup \bB_1^n$.
Define 
$$
K:= \bB_{R_1}^n\times\bB
$$
so that $(K_0\times\bB)\cup K_1$ is a subset of $K$.
Claim \ref{claim:K1} implies that 
the mapping $\psi|_\cW$ is a proper submersion over 
$\varphi^{-1}(\bB) \setminus K$. Ehresmann Fibration Theorem for a manifold 
with boundary applies. Therefore we have obtained the announced result 
of constancy of the $C^1$ type. Since 
$$
\psi^{-1}(\R\times\by) \cap \cW = \cW_\by \times\by
$$
we conclude that for all $\by \in \bB^\circ$, the links 
$\lk^\infty (\cW_\by)$ and $\lk^\infty (\cW_\bc)$ are $C^1$-diffeomorphic.

\smallskip
To get the result for the levels of $\omg$ instead of the family 
$(\cW_\by)_\by$, there is just to follow the same scheme of demonstration as 
the previous one, only modifying $\psi$ to become $\bp \mapsto (|\bp|^2,\by)$.
\end{proof}
%
%
%
%
%
%
%
%
%
%
%
%
%
%
%
%
%
%
%
%
%
%
%
%
%                          *******************************************************************
%
%
%
%
%
%
%
%
%
%
%
%
%
%
%
\section{Hyperplane sections and Malgrange-Rabier Condition}\label{section:plane-section-malgrange}
We start with the following simple lemma: 
\begin{lemma}[see also \cite{DutertreAdvGeo}]\label{lem:linear-section}
Let $S$ be a connected definable $C^2$ sub-manifold of $\R^q$ of 
positive codimension. 
For each $k\geq \dim S$, there exists an open dense and definable 
subset $\Omega_S^k$  of $\bG(k,q)$ such that $S \cap P$ is transverse (but 
possibly at the origin) for any $k$-plane $P$ of $\Omega_S^k$.
\end{lemma}
\begin{proof}
 Let $\bbo$ be the origin of $\R^q$ and let $S^* := S\setminus \bbo$. 
%If $P$ is the direction of a linear 
%subspace of $\R^q$, let $\la P \ra$ be the corresponding linear subspace.
Let $k \geq \dim S$. Let us consider the following subset
$$
B_k := \{(\bx,P) \, \in \, S^* \times \bG(k,q)\, :\,  \bx\in P\}.
$$
It is a closed definable $C^2$ sub-manifold of $S^*\times\bG(k,q)$ 
 of dimension $\dim S + (k-1)(q-k)$ since it  is
a fibre bundle over $S^*$ with model fibre $\bG(k-1,q-1)$. 
Let $\pi_k$ be the projection $B_k\to\bG (k,q)$. It is obviously a $C^2$
definable mapping. The subset $\Delta = \pi_k(\crit(\pi_k))$ of its 
critical values is definable in $\bG(k,q)$ and of positive codimension. 
The complement $\bG(k,q)\setminus \clos(\Delta)$ is the desired 
looked for $\Omg_S^k$. 
\end{proof}
Of course the proof works for any $k$, but for $q-k \geq \dim S$, we have 
$\pi(B_k)= \Delta$, so that the fibre of $\pi_k$ over $P\notin \Delta$ is empty.
%%A similar result also holds if we were to work with the Grassmann space
%%of affine subspaces of $\R^q$ instead of that of vector subspaces.

\bigskip
Since any linear subspace $P$ of dimension $k-s$ of $\R^n$ gives rise to a 
unique linear subspace $P\times\Rs$ of $\Rn\times\Rs$ of 
dimension $k$, we check that the following subset
$$
\Pi (k) := \{P\times\Rs \in \bG(k,n+s)\, : \, P\in \bG(k-s,n)\}.
$$
is a non-singular projective sub-variety of $\bG(k,n+s)$ isomorphic 
to $\bG(k-s,n)$.
%%\\
%%Similarly, given  a vector subspace $E\times F$ of finite dimension, 
%%we define $\Pi(l,E) \subset \bG(l+\dim F,E\times F)$, embedding 
%%$\bG(l,E)$ as done above when $E=\Rn$ and $F=\Rs$.
%
%
%%The first step to the main result of this section is 
%%the following.
%%%
%%%
%%\begin{lemma}\label{lem:M-R-hyperplane}
%%Let $\bc\in\Rs\setminus K(\vp)$. There exists a definable open dense subset 
%%$\Omg_\bc^{n+s-1}$ of $\Pi(n+s-1)$ consisting of hyperplanes $H\times\Rs$ 
%%such that the value $\bc$ does not lie in $K(\vp|_{H\times\Rs})$.
%%\end{lemma}
%
%

\medskip\noindent
The first main new result of the paper is the
following:
\begin{theorem}\label{thm:M-R-plane-sections}
Let $\bc\notin K(\vp)$. For every $k\geq n+s-\dim W$, there exists a 
definable open dense subset $\Omg_\bc^{k}$ of $\Pi(k)$ consisting of $k$-vector 
planes $P\times\Rs$ such that the value $\bc$ does not lie in 
$K(\vp|_{P\times\Rs})$.
\end{theorem}
\begin{proof}
Observe that for any vector $p$-plane $P$ of $\Rn$ with $p\geq 1$, we have
$$
\beta^{-1}(P) = ]0,+\infty[\times \bS(P) \subset \bMbar.
$$
Let us define the "$p$-plane" induced from $P$ in $\bMbar$
$$
P^+ := [0,+\infty[\times \bS(P) \subset \bMbar.
$$

\medskip\noindent
Let $\bc$ be a (regular) value of $\vp$. 
The set of accumulation points at infinity of the value $\bc$ of the mapping 
$\vp$ is the closed, definable subset of $\bS^{n-1}\times \bbo 
\subset \bS^{n+s-1}$ defined as
\begin{equation}\label{eq:xinftyc}
W_\bc^\infty :=\left\{\bu\in\bS^{n+s-1} \, ; \, \exists \, (\bw_k)_{k\in\N} 
 \subset  W \, : \, \bw_k \to \infty, \, \by_k \to \bc, \, 
\frac{\bw_k}{|\bw_k|} \to \bu\right\}.
\end{equation}
We write $W_\bc^\infty = A_\bc^\infty \times \bbo$.
%% for $A_\bc^\infty$ the projection of $W_\bc^\infty$ onto $\bS^{n-1}$.
We recall that $Z = \beta_s^{-1}(W)$ and that $Z^\infty = \clos (Z) \cap 
\bMbar_s^\infty$. Observe that 
$$ 
Z_\bc^\infty := Z^\infty\cap 0\times\bS^{n-1}\times \bc = 0\times A_\bc^\infty
\times\bc.
$$
We recall that $\bz = (r,\bu,\by) \in \bMbar_s$, and $\br_Z (\bz) = r$ 
for $\bz\in Z$. 
We need the following intermediary result:
\begin{claim}\label{claim:aF-at-c}
Let $\bc\in \Rs\setminus  K(\vp)$. There exists a definable 
stratification of $Z_\bc^\infty$ such that for any point $\bu$ of 
$A_\bc^\infty$ and for any sequence $(\bz_k)_{k\in\N}$ of $Z$ 
such that (i) $r_k \to 0$, (ii) $\by_k \to \bc$, (iii) $\bu_k \to \bu$, 
and (iv) $T_{\bz_k} \br_Z \to R$, we have
$$
T_\bu S \subset R,
$$
where $S$ is the stratum of $Z_\bc^\infty$ containing $0\times\bu\times\bc$. 
\end{claim}
\begin{proof}[Proof of the claim]
The function $\br_Z$ extends continuously and definably to $0$ on $Z^\infty$.  
Following \cite{Loi2,Loi3}, it can be definably stratified  
($\arel$). We can further ask that $0\times A_\bc^\infty \times \bc$ is a 
union of strata. Any stratum of $0\times A_\bc^\infty \times \bc$ is of 
the form $0\times S \times \bc$, for some sub-manifold $S$ of $\bS^{n-1}$.
%%
%%Let $(\bz_k)_k$ be a sequence of points of $\bMbar_s$ converging to 
%%$\bz_\infty = (0,\bu,\bc)\in  \bMbar_s^\infty$ such that 
%%$T_{\bz_k} \br_Z \to R$  and $T_{\bz_k} Z \to T$. If 
%%$\cS = 0\times S \times \bc$ is the stratum of
%%$0\times A_\bc^\infty \times \bc$ containing $\bz_\infty$, then
%%$T_{\bz_\infty} \cS \subset R \subset T$, and thus
%%$$
%%T_\bu S \subset Q'\cap (\R\times T_\bu\bS^{n-1} \times 0).
%%$$
%%Since $\bz_k = (\br_k,\bu_k, \by_k)$, let $\bx_k = \br_k^{-1}\bu_k$, and let 
%%$\bw_k = (\bx_k,\by_k) = \beta_s (\bz_k)$. 
%%We deduce that
%%$$
%%T_{\bw_k} W\to Q = \lim_k D_{\bz_k} \beta_s \cdot T_{\bz_k} Z,
%%$$
%%Let $\bz = (\bx,\by)$ be a point of $W$ with $\bx\neq 0$. 
%%Let $\bu = \frac{\bx}{|\bx|}$. Any vector $\bv$ of $T_\bw W 
%%\subset\Rn\times\Rs$ decomposes as the orthogonal sum
%%$$
%%\bv = v\bu\oplus\bv_\bS\oplus\bv_s \in \R\bu\oplus T_\bu\bS^{n-1}\oplus\Rs.
%%$$
%%For such a vector $\bv$ we find
%%$$
%%D_\bw \beta_s^{-1} \cdot \bv = (-\br^2 v, \br \bv_\bS ,\bv_s).
%%$$
%%We can also assume that $T_{\bw_k} \rho \to T$. 
%%Since $\bc\notin K(\vp)$, we find that $ \br_k^{-1} D_{\bw_k} \rho\cdot\bv_k 
%%\to 0$, for a sequence of unit vectors $(\bv_k)_k$ if and only if 
%%$\dlt_{n+s} (\bv_k,T_{\bw_k} \rho) \to 0$.  We deduce
%%$$
%%\lim_k D_{\bw_k} \beta_s^{-1} \cdot T_{\bw_k} \rho = Q'\cap(\R\times 
%%T_\bu \bS^{n-1}\times 0) \supset 0 \times T_\bu S \times 0
%%$$
%%and therefore $T_\bu S \subset T$.
\end{proof}
Let us write again $W_\bc \times \bc := \vp^{-1}(\bc)$. It is more convenient
to work in a neighbourhood of $\bMbar_s^\infty$, more precisely
nearby $Z_\bc^\infty$.
%%Let again $Z := \beta_s^{-1}(W)$. The analogue of $W_\bc^\infty$ is simply 
%%$0 \times W_\bc^\infty \subset \bMbar_s^\infty$, that is $0 \times 
%%A_\bc^\infty \times\bc$.

%%For any sub-manifold $S$ of $\Rn$ as in Lemma \ref{lem:linear-section},
%%let $\Omg_S^{k-s}$ be the corresponding definable open dense subset of 
%%$\bG(k-s,n)$.
%%
Let $S_1,\ldots,S_l$ be the strata of $A_\bc^\infty$ obtained from Claim 
\ref{claim:aF-at-c}. 

Let $k\leq n+s$ be any integer such that $k + \dim W - s \geq n$. 
Let us consider the following definable open dense subset of $\bG(k-s,n)$
$$
\Omg_\bc^k :=\cap_{j=0}^l 
\Omg_{S_j}^{k-s},
$$
where $\Omg_{S_j}^{k-s}$ is the definable open dense subset of 
$\bG(k-s,n)$ of Lemma \ref{lem:linear-section} corresponding to $S_j$.

\smallskip\noindent
Assume that $\dim W -s \geq n-k+1$.

Let $P$ be a $(k-s)$-plane of $\Omega_\bc^k$.
Let $(\bz_m)_m$ be a  sequence of $P^+\times\Rs\setminus Z^\infty$ such that 
$\bz_m \to \bz_\infty = (0,\bu,\bc) \in Z_\bc^\infty$. We can assume that 
$$
T_{\bz_m} \br_Z \to R, \;\; {\rm and} \;\; 
\bu_m \to \bu \in  (0\times A_\bc^\infty \cap P^+)
\times \bbo.
$$
Since $\bu\in P^+$, we deduce that $P = \R\bu \oplus T_\bu \bS(P)$.
\\
All the computations which will follow are done in 
$$
T_{\bz_\infty}\bMbar_s = \R\times T_\bu\bS^{n-1}\times\Rs = 
\R\times\R^{n-1}\times\Rs.
$$
Therefore we write $P$ for $\R\times T_\bu\bS(P) = T_{(0,\bu)} P^+ \subset 
\R\times\R^{n-1}$. 
\\
By hypothesis $R = 0\times R_1\subset 0\times\R^{n-1}\times\Rs$
and the space $R$ projects surjectively onto $\Rs$
(see Definition \ref{def:regular-infinity} and Lemma 
\ref{lem:rabier-surjective}).
We want to show that $R\cap P\times\Rs$ projects surjectively onto $\Rs$. 
Let 
$$
0\times K_R\times \bbo := R\cap (\R\times\R^{n-1})\times\bbo = 
\ker (R\mapsto\Rs).
$$ 
We can assume that $K_R = \R^p\times\bbo\subset \R^{n-1}$ where $p = \dim W - 
s- 1$. By hypothesis, denoting by $S$ the stratum of $A_\bc^\infty$ 
containing $\bu$, we have 
$$
T_\bu S \subset K_R, \;\; {\rm and} \;\; P + T_\bu S = \R\times\R^{n-1}.
$$
Since $T_\bu S$ is contained in $0\times\R^{n-1}$, the $k$-plane 
$P$ is not. Let $0\times P_1 := P\cap 0\times\R^{n-1}$. Since we find 
$$
K_R + P_1 = \R^{n-1}, 
$$
we deduce the following key fact 
$$
\bbo\times\R^{n-1-p} \subset P_1.
$$
Let $N_R := K_R^\perp \cap R$ be the orthogonal complement of $K_R$ in $R$, thus 
the projection to $\Rs$ when restricted to $N_R$ is an isomorphism. 
Since the space $N_R$ is contained in $\bbo\times \R^{n-1-p}\times\Rs$, we deduce that
$$
N_R =  (0\times P_1\times \Rs)\cap N_R \subset 
(P\times\Rs)\cap R.
$$
In other words $R\cap (P\times\Rs)$ projects surjectively onto $\Rs$.

\smallskip\noindent
Assume $\dim W -s = n-k$.
For each  generic $k$-plane $P$ of $\Omg_\bc^k$, the fibre 
$(\vp|_{P\times\Rs})^{-1}(\bc)$ is finite. The definition of $W_\bc^\infty$ 
prohibits, when it is not empty, that $\bc$ be a properness value
of $\vp|_{P\times\Rs}$, therefore $\bc$ lies in $K(\vp|_{P\times\Rs})$ if 
$W_\bc^\infty$ is not empty. In such a case, the arguments used in the 
previous case show the existence of a surjective mapping from a space of 
dimension $s-1$ onto $\Rs$.
\end{proof}
Let $F: X\to\Rs$ be a $C^2$ definable mapping over a closed connected 
$C^2$ definable sub-manifold $X$ of $\R^n$ with $\dim X\geq s$. The original
goal of this section, consequence of Theorem \ref{thm:M-R-plane-sections}, is 
the following result:
\begin{corollary}\label{cor:maps-M-R-plane-sections}
Let $\bc$ be a value not in $K(F)$.
For every $k\geq n - (\dim X-s)$, there exists a definable and dense open 
subset $\cV_\bc^k$ of $\bG(k,n)$ such that for every plane $P$ of 
$\cV_\bc^k$ the value $\bc$ does not lie in $K(F|_P)$.
\end{corollary}
\begin{proof}
Let $W$ be the graph of $F$ and let $\vp : W \to \Rs$ be the restriction
to $W$ of the projection $\Rn\times\Rs \to \Rs$.
The regular mapping 
$$
\Pi(k+s) \, {\buildrel \iota_n\over\mapsto} \, \bG(k,n), \; P \, \mapsto \, 
P_n := P\cap \Rn\times 0
$$
is an isomorphism.
Let $\bc\in\Rs$ which does not lie in $K(\vp) = K(F)$.
For any $l \geq (n+s) - (\dim W -s)$, 
let $\cU_\bc^{l+s}$ be the open dense subset of 
Theorem \ref{thm:M-R-plane-sections}. The image 
$$
\cV_\bc^k := \iota_n(\cU_\bc^{k+s})
$$
is definable, open and dense in $\bG(k,n)$.
\end{proof}
%
%
%%%
%%%
%%\begin{remark}\label{rmk:plane-complex}
%%The properties of being a regular, a (MR)-regular value
%%for a given complex mapping is inherited from its underlying real structure.
%%
%%Assume that we work with a complex setting, that is $W$ is a non-singular 
%%irreducible affine sub-variety of $\C^n\times\C^s$ and the mapping 
%%$\vp: W \mapsto\C^s$ is again the projection onto the second factor $\C^s$
%%(see also \cite{JeKu}).
%%(Compactifying in this setting means taking the projective one or the product 
%%of the projective compactifications of the  ambient factors.) 
%%In such a complex context, we would most certainly consider restriction of 
%%$\vp$ to complex plane sections of the family of projection onto $\C^n$ of 
%%the levels of $\vp$.
%%All results presented in this section, namely Lemma \ref{lem:linear-section},
%%Lemma \ref{lem:aF-at-c}, 
%%%%Lemma \ref{lem:M-R-hyperplane}, 
%%Theorem 
%%\ref{thm:M-R-plane-sections} and Corollary \ref{cor:maps-M-R-plane-sections}
%%admit straightforward complex version for complex plane sections 
%%and the proofs are almost word for word the same as those of the real setting. 
%%\end{remark}
%%%
%%%
%
%
%
%
%
%
%
%
%
%
%
%         **************************************************
%
%
%
%
%
%
%
%
%
%
%
%
%
%
\section{Lipschitz-Killing Measures}\label{section:LKM}

We present very briefly in this section the Lipschitz-Killing measures of a 
definable set in an o-minimal structure following Br\"ocker and Kuppe's
approach \cite{BroeckerKuppe}. They are the essential ingredients to define
the functions $\Lbd_k^\infty(-)$ (Definition \ref{def:chi-lbd-infty}),
bricks of the general Gauss-Bonnet Formula presented in Theorem
\ref{thm:GaussBonnet}.

\bigskip
We start with a few reminders about the extrinsic geometry of 
sub-manifolds of the Euclidean space $\Rn$.

\smallskip
Let $Z$ be a $C^2$ connected orientable sub-manifold of $\Rn$ equipped
with the restriction of the Euclidean metric tensor.
Let $\bx$ be a point of $Z$. 
Let $\bS(N_\bx Z)$ be the unit sphere of the normal space $N_\bx Z = 
T_\bx Z^\perp$ in $T_\bx \Rn = \Rn$.
For $\bv$ a given vector of $\bS(N_\bx Z)$, let $\bv^*$ 
be the linear form over $T_\bx \Rn$ defined by the scalar product with $\bv$.

Let $\nb$ be the covariant differentiation in $\Rn$ w.r.t. the 
Euclidean metric tensor.

Of fundamental importance to define the Lipschitz-Killing curvatures 
of $Z$ at any of its points is the family of second fundamental forms 
$(II_{\bn})_{\bn\in\bS(NZ)}$, where $\bS(NZ)$ is the unit sphere bundle
of the normal bundle $NZ$ of $Z$ in $T\Rn|_Z$.
Given $\bn = (\bx,\bv)$ in $\bS(NZ)$, we recall that $II_\bn$, \em the second 
fundamental form in the direction $\bv$, \em is the symmetric bilinear 
form over $T_\bx Z$ defined as
$$
II_\bn(\bu_1,\bu_2) =-\langle \nabla_{\bu_1} \nu,\bu_2 \rangle,
$$
where $\bu_1,\bu_2$ are vectors of $T_\bx Z$ and $\nu$ is any $C^1$ local 
extension of $\bv$ normal to $Z$ at $\bx$.

For each $l=0,\ldots, \dim Z$, let  
$\sigma_l^Z (\bn)$ be the $l$-th elementary symmetric function of the 
eigenvalues of $II_\bn$ when considered as a symmetric endomorphism of 
$T_\bx Z$.

\bigskip
Let $X$ be a closed definable subset of $\mathbb{R}^n$ equipped with  
$\mathcal{S}=\{ S_a \}_{a \in A}$, a finite $C^2$ definable Whitney 
stratification.  
Let $S$ be a stratum $S_a$ of dimension $d_S$. Let $\bx$ be a point of 
$S$ and let $\bv$ be a unit vector normal to $S$ at $\bx$. We recall 
the definition of the following index
$$
{\rm ind}_{\rm nor}(\bv^*,X,\bx)= 1-\chi \left( X \cap N_\bx  \cap 
\bB^n (\bx,\ve) \cap \{ \bv^*= \bv^*(\bx)-\delta \} \right),
$$
where $0 < \delta \ll \ve \ll 1$ and $N_\bx$ is a normal (definable) 
slice to $S$ at $\bx$ in $\Rn$.

\medskip
For each $k = 0,\ldots,n,$ we define the function 
$\lambda_k^S : S \rightarrow \mathbb{R}$ as
$$
\lambda_k^S(\bx) : = 
\left\{
\begin{array}{rcl} 
\displaystyle{
\frac{1}{s_{n-k-1}} \int_{\bS(N_\bx S)}} {\rm ind}_{\rm nor}(v^*,X,\bx) 
\sigma_{d_S-k}^S (\bx,\bv) \rd\bv & {\rm if} & 0 \leq k \leq d_S \\
0 & {\rm if} & d_S + 1 \leq k \leq n,
\end{array}
\right.
$$
where $s_l$ is the $l$-volume of the unit Euclidean sphere 
$\bS^l$ of $\R^{l+1}$.
\\
If $S$ has dimension $n$ then for all $\bx \in S$, we find 
$$
\lambda_0^S =\cdots=\lambda_{n-1}^S \equiv 0 \;\; {\rm and} \;\; \lambda_n^S 
\equiv 1.
$$ 
If $S$ has dimension $0$ then 
$$
{\rm ind}_{\rm nor}(\bv^*,X,\bx)= {\rm ind}(\bv^*,X,\bx):=
1-\chi \left( X \cap \bB^n(\bx,\ve) \cap \{ \bv^*= \bv^*(\bx)-
\delta \} \right),
$$
and we set
$$
\lambda_0^S(\bx)= \frac{1}{s_{n-1}} \int_{\bS^{n-1}}  {\rm ind}(\bv^*,X,\bx) 
\rd\bv, \;\; {\rm and} \;\; 
\lambda_k^S \equiv 0, \;\; {\rm once} \;\; k\geq 1.
$$

\begin{definition}
Let $k\in \{0,\ldots,n\}$. The \em $k$-th Lipschitz-Killing measure 
$\Lbd_k(X,-)$ of $X$ \em is defined as follows:
$$
U \mapsto 
\Lbd_k(X,U)= \sum_{a \in A} \int_{S_a \cap U} \lambda_k^{S_a} 
(\bx) \rd\bx,
$$
where $U$ is any bounded Borel subset of $X$.
\end{definition}
\noindent
Denoting by $d$ the dimension of $X$, we obviously have
$$
\Lbd_{d+1}(X,-)= \cdots=\Lbd_n(X,-) \equiv 0,
$$ 
and for any bounded Borel subset $U$ of $X$ we get
$$
\Lbd_{d}(X,U)= \mathcal{H}_{d}(U),
$$ 
where $\mathcal{H}_{d}$ is the $d$-th dimensional Hausdorff measure in 
$\mathbb{R}^n$. 

\bigskip
An exhaustive family of compact subsets $(K_R)_{R >0}$ of $X$ 
is an increasing - for the inclusion - family of compacts of $X$ covering $X$: 
$$
\bigcup_{R>0} K_R = X.
$$

\medskip
We introduce some new notations in order to present short formulae. Let 
$g_n^l$ be the volume of the Grassmann manifold $\bG(l,n)$ of 
vector $l$-planes of $\Rn$ when equipped with the Euclidean metric
(see Section \ref{section:linear-algebra}). We 
start with the following two sets of numbers:
\begin{definition}\label{def:chi-lbd-infty}
Let $X$ be a closed definable subset of $\Rn$.
For each $l=0,\ldots,n,$ let $\chi_l^\infty(X)$ be defined as
\begin{equation}\label{eq:chi-l-infty}
\chi_l^\infty(X) := 
\displaystyle{
\frac{1}{2g_n^l}\int_{\bG(l,n)} \chi\left(\lk^\infty(X\cap P) \right)\rd P,
}
\end{equation}
and let $\Lbd_l^\infty(X)$ be \em the $l$-th Lipschitz-Killing invariant
of $X$ at infinity \em defined as
\begin{equation}\label{eq:lbd-l-infty}
\Lbd_l^\infty(X) := \lim_{R \rightarrow + \infty} 
\frac{\Lbd_l(X,X\cap \bB_R^n)}{b_l R^l},
\end{equation}
where $b_l$ is the volume of the Euclidean unit ball $\bB_1^l$.
\end{definition}
\noindent
Observe that for $X\subset\Rn$ we have the obvious equality
\begin{equation}\label{eq:chi-infty-max}
\chi(\lk^\infty(X)) = 2\chi_n^\infty(X).
\end{equation}

\medskip
When $X$ is a closed definable connected sub-manifold, note that 
$\chi_k^\infty(X) = 0$ whenever $\lk^\infty(X\cap P)$ has odd
dimension, that is when $\dim X - (n-k)$ is even. 
When $n-k\geq \dim X$ we also find that 
$$
\chi_k^\infty (X) = 0,
$$
since the link $\lk^\infty(X\cap P)$ is empty for $P$ lying in a definable 
open dense subset of $\bG(k,n)$.

\medskip
We recall now several Gauss-Bonnet type formulas for a closed definable 
set $X$ subset of  $\Rn$ established by the first author in
\cite{DutertreAdvGeo,DutertreGeoDedicata2012} and which relate the numbers 
$\chi_l^\infty(-)$ and $\Lbd_l^\infty(-)$.
\begin{theorem}[\cite{DutertreAdvGeo,DutertreGeoDedicata2012}]\label{thm:GaussBonnet}
1) The limit $$
\Lambda_0^\infty(X) := \lim_{R \to +\infty} \Lambda_0(X,X\cap K_R)
$$ 
exists and does not depend on the choice of the exhaustive family of compact 
subsets $(K_R)_{R >0}$. More precisely the following equality holds: 
\begin{equation}\label{eq:GB-0}
\Lambda_0^\infty(X)= \chi(X) -  \chi_n^\infty(X) - \chi_{n-1}^\infty(X).
\end{equation}

\medskip\noindent
2) For each $k= 1,\ldots,n-2,$ we furthermore have:
\begin{equation}\label{eq:GB-n-2}
\Lbd_k^\infty(X)
= - \chi_{n-k-1}^\infty(X) + \chi_{n-k+1}^\infty(X),
\end{equation}
and for $l=n-1,n,$ we have
\begin{equation}\label{eq:GB-n-1-n}
\Lbd_l^\infty(X) = \chi_{n-l+1}^\infty(X).
\end{equation}
\end{theorem}
\begin{remark}\label{rmk:matrix}
Consider the following two vectors of $\R^{n+1}$
$$
\Lbd_*^\infty(X) = (\Lbd_0^\infty(X),\ldots,\Lbd_n^\infty(X))
\;\; {\rm and} \;\;
\chi_{-*}^\infty(X) = (\chi(X),\chi_n^\infty(X),\ldots,\chi_1^\infty(X)).
$$
As seen in Theorem \ref{thm:GaussBonnet}, 
both vectors carry the same information about $X$ at infinity. Precisely,
there exists a triangular superior matrix $L$, depending only on $n$, with 
coefficients in $\{-1,1\}$ and only with $1$ on the diagonal such that
$$
\Lbd_*^\infty(X) = L\cdot \chi_{-*}^\infty(X). 
$$
\end{remark}
%
%
%
%
%
%
%
%
%
%
%
%
%
%
%
%         **************************************************
%
%
%
%
%
%
%
%
%
%
%
%
%
%
\section{The case of sub-manifolds with boundary}\label{section:smfd-bdry}

We describe the Lipschitz-Killing measures when $(X,\dd X)$ is a closed $C^2$ 
definable sub-manifold with boundary of $\Rn$ of dimension $d$. In this 
case, the partition 
$X = X^\circ \sqcup \partial X$, where $X^\circ = X \setminus \partial X$, is 
a $C^2$ definable Whitney stratification of $X$, to which we can apply the 
construction of Section \ref{section:LKM}.

\medskip 
Let $\bx \in X^\circ$. If $d < n$ then for $k = 0,\ldots,d$, we have
$$
\lambda_k^{X^\circ}(\bx)= \frac{1}{s_{n-k-1}} \int_{\bS(N_\bx X^\circ)} 
\sigma_{d-k}^{X^\circ} (\bx,\bv) \rd\bv.
$$
We get the following identities of the extrinsic geometry 
of Euclidean sub-manifolds
$$
\lambda_k^{X^\circ}(\bx)= \frac{1}{s_{n-k-1}} K_{d-k}({X^\circ},\bx),
$$ 
where $K_{d-k}(X^\circ,-)$ is the $(d-k)$-th Lipschitz-Killing curvature 
of $X^\circ$. We recall that 
$$
d-k \;\; {\rm odd}  \;\; \Longrightarrow \; K_{d-k}(X^\circ,-) \equiv 0.
$$
We recall that if $d=n$, then $\lambda_n^{X^\circ} \equiv 1$ and 
$\lambda_0^{X^\circ} =\cdots=\lambda_{n-1}^{X^\circ}  \equiv 0$. 

\medskip
Let $\by \in \partial X$ and let $\nu_\by$ be the unit vector tangent to 
$X$ at $\by$, normal to the boundary $\partial X$ and pointing inwards.
For $\bv \in \bS^{n-1}$ we recall that the following alternative holds
$$
{\rm ind}_{\rm nor}(\bv^*,X,y) = 
\left\{
\begin{array}{rcl}
1 & {\rm if} & \la \bv, \nu_\by \rangle >0 \\
0 & {\rm if} & \langle \bv, \nu_\by \rangle <0
\end{array}
\right.
.
$$
Let $\bS(N_\by \dd X)^+$ be the following open half-sphere:
$$
\bS(N_\by \dd X)^+ = \left\{ \bv \in \bS(N_\by \dd X) \; : \; 
\langle \bv, \nu_\by \rangle >0 \right\}.
$$
Therefore, for $k = 0,\ldots,d-1,$ we can write
$$
\lambda_k^{\partial X} (\by) =  \frac{1}{s_{n-k-1}} \int_{\bS(N_\by \dd X)^+}
\sigma_{d-1-k}^{\dd X} (\by,\bv) \rd\bv.
$$
The case of a definable sub-manifold with boundary of dimension $d$ yields 
the following explicit description of the Lipschitz-Killing measures.
For any bounded Borel subset $U$ of $X$ we find
\begin{equation}\label{eq:LKM-mfd-bdr-d}
\Lambda_d(X,U)= \mathcal{H}_d(U). 
\end{equation}
When $d<n$ and $k  = 0,\ldots,d-1$, we find
\begin{equation}\label{eq:LKM-mfd-bdr-dk} 
\Lambda_k(X,U)= \int_{X^\circ \cap U} \lambda_k^{X^\circ}(\bx)\rd\bx 
+ \int_{\dd X \cap U} \lambda_k^{\partial X} (\by) \rd\by.
\end{equation}
If $d=n$ and $k  = 0,\ldots,n-1$,
\begin{equation}\label{eq:LKM-mfd-bdr-nk}
\Lambda_k(X,U)= 
\int_{\dd X \cap U} \lbd_k^{\dd X}(\by) \rd \by. 
%%= 
%%\frac{1}{s_{n-k-1}}\int_{\dd X \cap U} \sgm_{n-1-k}^{\dd X}(\by) \rd \by.
\end{equation}

%
%
%
%
%
%
%
%
%
%
%
%
%
%         **************************************************
%
%
%
%
%
%
%
%
%
%
%
%
%
%
\section{Definable families and continuity of Lipschitz-Killing 
curvature densities at infinity}\label{section:cont-curv-main}

We return to the setting of the previous sections: let $W$ be a $C^2$ 
connected sub-manifold of $\Rn \times \Rs$ which is also a closed subset of 
$\Rn$, definable in $\cM$, and let $\vp : W \to \Rs$ be the restriction
to $W$ of the canonical projection 
on $\Rs$. Let $d =\dim W$. Hence for $\by$ any regular value of $\vp$ 
either the level $\vp^{-1}(\by)$ is empty or is a $C^2$ definable 
sub-manifold of dimension $d-s$ of $\Rn$,
of the form 
$$
\vp^{-1}(\by) = W_\by \times \by \subset \Rn\times\Rs.
$$
Instead of working with the  family of levels $(\vp^{-1}(\by))_\by$ we will
 work with 
the  family $(W_\by)_\by$ of the projections onto $\Rn$ of the levels
of $\vp$, a definable family of closed subsets of $\Rn$.

\bigskip
Let $\cV_1,\ldots,\cV_\beta$ be the connected components of 
$\Rs\setminus K(\vp)$ for which $\vp^{-1}(\cV_b)$ is not empty,.
For each $b=1,\ldots,\beta$, let $\cU_{b,1},\ldots,\cU_{b,\aph_b}$ be the 
connected components of $\vp^{-1}(\cV_b)$. 
For each $\by \in \cV_b$ and each $a = 1,\ldots,\aph_b,$ let
\begin{equation}\label{eq:Wbya}
W_\by^a \times \by := W_\by \times \by \cap \cU_{b,a}.
\end{equation}
Each such sub-manifold $W_\by^a$ is not empty, closed and connected, by 
Corollary \ref{cor:trivialisation}.

\medskip
We present in this section now the looked for continuity results of the 
functions $\by \mapsto \Lbd_k^\infty(W_\by)$ nearby a regular value which
is also a (MR)-regular value of this definable family.
\begin{proposition}\label{ContLinkMappings}
For each $k = 0, \ldots, d-s$, 
%%the functions $\by \to \chi_k^\infty(W_\by)$ and $\by \to 
%%\Lbd_{n-k-1}^\infty(W_\by)$ are continuous over $\Rs\setminus K(\varphi)$. 
for each $b=1,\ldots,\beta,$ and each $a=1,\ldots,\aph_b,$
the functions 
$$
\by \mapsto \chi_{n+1-k}^\infty(W_\by^a),\;\; {\rm and} \;\; 
\by \mapsto \Lbd_{d-s-k}^\infty(W_\by^a)
$$
are continuous over $\cV_b$,
(where $\chi_{n+1}^\infty(-)$ is the null function over the subsets of $\Rn$).
\end{proposition}
\begin{proof}
Theorem \ref{thm:GaussBonnet} and Remark \ref{rmk:matrix} guarantee that it is 
sufficient to prove it for the functions $\by \mapsto 
\chi_{n+1-k}^\infty(W_\by)$. Observe that the mapping $\by \mapsto 
\chi(W_\by^a)$ is constant on 
$\cV_b$ for each $a=1,\ldots,\aph_b,$ by Corollary \ref{cor:trivialisation}.

\medskip\noindent
For a closed definable subset $X$ of $\R^n$, let us write 
$$
\chi^\infty(X) := \chi(\lk^\infty(X)).
$$
$\bullet$ Assume first that $W_\by$ is connected for $\by\notin K(\vp)$.
%%For each $k = n-(d-s)+1,\ldots,n,$ consider the following 
%%(well defined) functions
%%$$
%%\lbd^k : \Rs\mapsto \R, \;\; \by \to \lbd^k(\by) := 2 g_n^k \cdot 
%%\chi_k^\infty(W_\by) = \int_{\bG(k,n)} \chi^\infty (W_\by\cap P)  \rd P.
%%$$
The case $k=0$ is obvious by definition, so
we treat first the case $k=1$.
%%, and thus find
%%$$
%%\lbd^n(\bc) = \chi^\infty (W_\bc).
%%$$
Let $\bc \in \R^s \setminus K(\varphi)$.
Proposition \ref{prop:link-infty-const} implies the existence of 
an open neighbourhood $U$ of $\bc$ in $\Rs$ such that 
$$
\by \in U \Longrightarrow 
2 \chi_n^\infty(W_\by) = 
\chi^\infty (W_\by) =  \chi^\infty (W_\bc).
$$
In other words, the mapping 
$\by \mapsto \chi_n^\infty (W_\by)$
is constant on each connected component of $\Rs \setminus K(\varphi)$. 

\medskip
We treat now the case $k =2,\ldots d-s$,
and $\bc \in \R^s \setminus K(\varphi)$. 
Theorem \ref{thm:M-R-plane-sections} ensures the existence of an open dense 
definable subset $\cU^k_\bc$ of $\bG(n+1-k,n)$ such that
$$
P \in \cU_\bc^k \Longrightarrow \bc\notin K(\vp|_{P\times\Rs}).
$$
For any $P \in \bG(n+1-k,n)$ and any $\by\in \Rs$ we find
$$
(\vp|_{P\times\Rs})^{-1}(\by) = (W_\by\cap P) \times\by \subset P\times\Rs.
$$
As for $k=1$, given $P \in \cU_\bc^k$, there exists an open neighbourhood 
$U_P$ of $\bc$ contained in $\Rs \setminus K(\vp|_{P\times\Rs})$ 
such that for all $\by \in U_P$, we find
$$
\chi^\infty (W_\by\cap P) = \chi^\infty (W_\bc \cap P),
$$
which we can write as 
$$
\lim_{\by \to \bc} \; \chi^\infty (W_\by \cap P) = \chi^\infty 
(W_\bc \cap P).
$$
Now Lebesgue Dominated Convergence Theorem provides  
\begin{eqnarray*}
\lim_{\by \to \bc} \int_{\bG_n^{n+1-k}} \chi^\infty (W_\by \cap P) \rd P
& = & \int_{\bG_n^{n+1-k}} \lim_{\by \to \bc} \chi^\infty (W_\by\cap P) \rd P
\\
& = & 
\int_{\bG_n^{n+1-k}} \chi^\infty (W_\bc \cap P) \rd P,
\end{eqnarray*}
namely the desired continuity result.

\medskip\noindent
$\bullet$ Let $b \in \{1,\ldots,\beta\}$ and let $\by\in\cV_b$.
For each $a = 1,\ldots,\aph_b$, the connected case implies that
each function $\by \mapsto \chi_{n+1-k}^\infty(W_\by^a)$ is continuous over
$\cV_b$. 
\end{proof}

\medskip
Let $X$ be a closed definable sub-manifold of $\R^q$ of dimension $d$.
For $i=0,\ldots,d$, let $K_i(X,\bx)$ be the $i$-th Lipschitz-Killing curvature
of $X$ at $\bx$. For $i\geq d+1$, we may define $K_i(X,-)\equiv 0$ if need be.
The global polar-like invariant avatars  we are looking for in our 
real present context are presented in the next definition.
\begin{definition}\label{def:kp-i-infty} 
Let $i \in \{0,\ldots,d\}$. The \em 
$i$-th Lipschitz-Killing  curvature density at infinity
of the closed definable sub-manifold $X$ of $\R^q$ of dimension $d$ 
\em is defined as 
$$
\kp_i^\infty(X) := \lim_{R \to +\infty} \frac{1}{R^{d-i}} \int_{X \cap 
\bB_R^n} K_i(X,\bx) \rd\bx.
$$
\end{definition}

\medskip
The second main result of the paper is the following
\begin{theorem}\label{ContCurvaturesMappings}
For each $i = 0,\ldots,d-s$, 
for each $b=1,\ldots,\beta,$ and each $a=1,\ldots,\aph_b,$
the functions 
$$
\by \to \kp_i^\infty(W_\by^a)
$$ 
are continuous over $\cV_b$. Thus the functions $\by\mapsto 
\kp_i^\infty(W_\by)$ are continuous on $\mathbb{R}^s \setminus K(\vp)$.
\end{theorem}
\begin{proof}
As we saw in the demonstration of Proposition \ref{ContLinkMappings},
we can work only with connected $W_\by$ as far as $\by$ lies in $\Rs\setminus 
K(\vp)$, which we will assume for the rest of the proof.

\medskip
We treat first the case $i=d-s$. Theorem \ref{thm:GaussBonnet} (see also
\cite[Theorem 5.6]{DutertreAdvGeo}) gives 
$$ 
\frac{1}{s_{n-1}} \kp_{d-s}^\infty (W_\by) = 
\chi (W_\by) - \chi_n^\infty(W_\by) - 
\chi_{n-1}^\infty(W_\by)
$$
for $\by$ any regular value. 
Applying Proposition \ref{ContLinkMappings} concludes this case.

\medskip
For $2 \le i \le d-s-1$ and $\by$ any regular value of $\vp$,
Theorem \ref{thm:GaussBonnet} (see also 
\cite[Theorem 4.1]{DutertreGeoDedicata2012}) gives 
$$
\frac{1}{s_{n-(d-s)+i-1} \cdot b_{d-s-i}} \kp_i^\infty(W_\by)
= - \chi_{n-(d-s)+i-1}^\infty(W_\by) + \chi_{n-(d-s)+i+1}^\infty(W_\by).
$$
We conclude again by Proposition \ref{ContLinkMappings}.

\medskip
For $i=0$ or $1$ and $\by$ any regular value of $\vp$, 
Theorem \ref{thm:GaussBonnet} gives again 
$$
\frac{1}{s_{n-(d-s)+i-1} \cdot b_{d-s-i}} \kp_i^\infty(W_\by) = 
\chi_{n-(d-s)+i+1}^\infty(W_\by).
$$
We remark that the term 
$$
\chi_{n-(d-s)}^\infty(W_\by)
$$
that appears in the equality concerning the curvature $K_1$ is in fact zero, 
because generically $W_\by \cap P$ is of dimension $0$, hence compact.
We conclude with the previous proposition again. 
\end{proof}
\begin{remark}\label{rmk:Ki-iodd)}
For odd $i$, the $i$-th Lipschitz-Killing curvature function $K_i(X,-)$ 
of a given sub-manifold $X$ is identically null, therefore the mapping 
in our context $\by \mapsto  \kp_i^\infty(W_\by)$
is the null mapping over $\Rs\setminus K(\vp)$.
\end{remark}
%
%
%
%
%
%
%
%
%
%
%
%
%
%
%
%         **************************************************
%
%
%
%
%
%
%
%
%
%
%
%
%
%
\section{Curvatures of hypersurfaces}\label{section:curv-hyper}

In this section, we study the special case of regular levels of a $C^2$ 
definable function, and express some curvature-like integrals over them as 
Lipschitz-Killing measures. 

\medskip
Let $f : \mathbb{R}^n \to \mathbb{R}$ be a $C^2$ definable function. 
We assume that $0$ is a regular value of $f$ taken by $f$, 
so that $Y=f^{-1}(0)$ is a non-empty $C^2$ hypersurface. 
We orientate $Y$ by the normal vector $-\nabla f|_Y$, 
with the convention that $(\xi_1,\ldots,\xi_{n-1})$ is a positive basis of 
$T_\bx Y$ if and only if $(\xi_1,\ldots,\xi_{n-1},-\nabla f (\bx))$ is
a positive basis of $\mathbb{R}^n$. 

\smallskip\noindent
The Gauss mapping $\nu_Y$ of $Y$ is the following mapping:
$$
\nu_Y  :  Y \, \to \, \bS^{n-1}, \;\; 
\bx \mapsto  -\frac{\nabla f }{\vert \nabla f \vert} (\bx) .
$$
Its derivative 
$$
D_\bx \nu_Y : T_\bx Y \mapsto T_{\nu(\bx)} \bS^{n-1} = T_\bx Y
$$ 
is a self-adjoint operator. 
The principal curvatures $k_1(\bx),\ldots,k_{n-1}(\bx)$ are the opposite of 
the eigenvalues of $D_\bx \nu_Y$. Let $\sigma_i^Y$ be the $i$-th elementary 
symmetric functions of the principal curvatures. Note that $\sigma_{n-1}^Y$ 
is the Gauss-Kronecker curvature and $\sgm_0^Y \equiv 1$. These elementary 
symmetric functions are 
related to the Lipschitz-Killing curvatures of $Y$ and the sub-level set 
$\cY := \{ f \le 0 \}$ in the  following way: Given $\bx \in Y$, 
since $Y$ is the boundary of $\cY$, we note that
$$
i \;\; {\rm even} \;\; \Longrightarrow \; 2\sigma_i^Y (\bx)=K_i(Y,\bx)
= s_i \cdot \lambda_{n-1-i}^Y (\bx).
$$
\begin{definition}\label{def:sgmi-infty}
Let $i \in \{0,\ldots,n-1\}$. The \em $i$-th symmetric principal 
curvature density at  infinity of the hypersurface $Y$ \em is 
defined as
\begin{equation}\label{eq:sgmi-infty}
\sgm_i^\infty (Y) := \lim_{R \to +\infty} \frac{1}{R^{n-1-i}}
\int_{Y \cap  \bB_R^n} \sigma_i^Y (\bx) \rd\bx.
\end{equation}
The density at infinity of $\cY$ is defined as
\begin{equation}\label{eq:thetan-infty}
\Tht_n^\infty (\cY) := \lim_{R \to +\infty} \frac{1}{R^n}
\cH_n(\cY \cap  \bB_R^n).
\end{equation}
\end{definition}
In this context, Equations \eqref{eq:LKM-mfd-bdr-nk} give
$$
\sgm_{n-1-i}^\infty  (Y) = s_{n-1-i}b_i\cdot\Lbd_i^\infty (\cY), \;\;
i=0,\ldots,n-1.
$$ 
The next corollary gives explicit relations between the numbers 
$\chi_l^\infty(-)$ of Equation \eqref{eq:chi-l-infty} and Equation 
\eqref{eq:chi-infty-max} and the numbers $\sgm_l^\infty(-)$ of Equation 
\eqref{eq:sgmi-infty}. It is an
application of the 
Gauss-Bonnet formulas of Theorem \ref{thm:GaussBonnet} for the levels of
the function $f$,
the description of the Gauss-Bonnet measures for manifolds with 
boundary and Equations \eqref{eq:LKM-mfd-bdr-nk} for the numbers 
$\sgm_i^\infty(-)$. 
\begin{corollary}\label{cor:sgmi-infty}
\begin{equation}\label{eq:GB-hyp-n}
\Tht_n^\infty(\cY) = \chi_1^\infty(\cY),
\end{equation}
and for $i=n-1$
\begin{equation}\label{eq:GB-hyp-n-1}
\frac{1}{s_{n-1}}\sgm_{n-1}^\infty (Y)  = 
\chi(\cY) - \chi_n^\infty(\cY) - \chi_{n-1}^\infty(\cY),
\end{equation}
and for $i=0,\ldots,n-2$,
\begin{equation}\label{eq:GB-hyp-i}
\frac{1}{s_ib_{n-i-1}}\sgm_i^\infty (Y) =  
- \chi_i^\infty(\cY) + \chi_{i+2}^\infty(\cY).
\end{equation}
\end{corollary}
\begin{remark}
We can orientate $Y$ by the normal vector $\nabla f|_Y$, with the convention 
that $(\xi_1,\ldots,\xi_{n-1})$ is a positive basis of $T_\bx Y$ if and only 
if $(\nabla f(\bx),\xi_1,\ldots,\xi_{n-1})$ is a positive basis of 
$\Rn$.
In this case, the Gauss mapping $\nu_H$ of $H$ is the mapping:
$$
\nu_Y  :  Y \, \to \, \bS^{n-1}, \;\; 
\bx \mapsto  \frac{\nabla f }{\vert \nabla f \vert} (\bx) .
$$
and the principal curvatures $k_1(\bx),\ldots,k_{n-1}(\bx)$ are the 
eigenvalues of $D_\bx \nu_Y$.
\end{remark}
%
%
%
%
%
%
%
%
%
%
%
%
%
%
%         **************************************************
%
%
%
%
%
%
%
%
%
%
%
%
%
%
%
%
\section{Continuity of curvature integrals and families of 
hypersurfaces}\label{section:cont-curv-hyper}
In this last section we expand and strengthen the study initiated 
in \cite{DuGr1}, addressing here the special case of families of hypersurfaces. 

\medskip
We work within the context of Section \ref{section:sub-levels}.
We are given a $C^2$ definable function $F:\Rn \times\Rs \to \R$ for
which $0$ is a regular value. Let again $\cW$ and $W$ be defined as
$$
\cW := \{\bp:F(\bp)\leq 0\}, \;\; {\rm and} \;\; W := \{\bp:F(\bp)=0\}.
$$
The mappings $\omg,\vp$ are the restriction of the projection
$\Rn\times\Rs \to\Rs$ to $\cW,W$, respectively.
Practically we work with the connected components of $W$, thus we further 
assume that $W$ is connected. For each $\by\in\Rs$, the
function $f_\by:\Rn\to\R$ defined as $\bx \mapsto F(\bx;\by)$ is also $C^2$
and definable. 
We recall that  
$$
\cW\cap\Rn\times\by = \cW_\by \times \by, \;\; {\rm and} \;\; 
\vp^{-1}(\by) = W_\by \times \by.
$$
Whenever $\by$ is a regular value of $\vp$, the hypersurface $W_\by= 
f_\by^{-1}(0)$ bounds the $C^2$ closed definable $n$-dimensional 
sub-manifold with boundary $\cW_\by := \{f_\by \le 0\}$.

Let $\cV_1,\ldots,\cV_\beta$ be the connected components of $\Rs
\setminus K(\vp)$ for which $\vp^{-1}(\cV_b)$ is not empty. For each $b=1,\ldots,\beta$, let $\cU_{b,1}, \ldots,
\cU_{b,\aph_b}$ be the connected components of $\omg^{-1}(\cV_b)$. For 
each couple $(b,a)$ and each $\by\in\cV_b$, let 
$$
\cW_\by^a := \cW_\by \cap \cU_{b,a}. 
$$
We obtain the following:
\begin{proposition}\label{prop:cont-link-hyper}
For each $b=1, \ldots,\beta,$ for each $a=1,\ldots,\aph_b$, and for 
$k = 1,\ldots,n$, the functions $\by \mapsto \chi_k^\infty(\cW_\by^a)$ 
are continuous over $\R\setminus K(\vp)$.
\end{proposition}
\begin{proof}
We can assume that $b=1$ and $\aph_1=1$ as well.
Thus $\cW_\by^a = \cW_\by$.

\smallskip
We treat first the case $k=n$. 
Proposition \ref{prop:link-infty-const-sub-level} shows that the following 
function is constant on $\cV_b$
$$
\by \mapsto \chi_n^\infty(\cW_\by) = \chi^\infty(\cW_\by).
$$

\smallskip
We assume that $1 \le k \le n$. Let $\bc \in \Rs \setminus K (\vp)$.  
Theorem \ref{thm:M-R-plane-sections} guarantees  that for almost all $P \in 
\bG(k,n)$, the value $\bc$ does not lie in $K_\infty (\vp|_{P\times\Rs})$. 
Proposition \ref{prop:link-infty-const-sub-level}
gives the existence of an open neighbourhood $V_P$ of $\bc$ in $\Rs$
such that for all $\by$ in $V_P$, 
$$
\chi^\infty (\cW_\by \cap P ) = \chi^\infty (\cW_\bc \cap P ).
$$ 
The proof ends applying Lebesgue Dominated Convergence Theorem, as was 
done in the proof of Proposition \ref{ContLinkMappings}. 
\end{proof}
For each $b=1,\ldots,\beta,$ let $\cU_{b,1},\ldots,\cU_{b,\aph_b}$ be the 
connected components of $\vp^{-1}(\cV_b)$. For $\by$ in $\cV_b$ and 
$a=1,\ldots,\aph_b$, let 
$$
\cW_\by^a := \cW_\by \cap \cU_{b,a}, \;\; {\rm and} \;\; 
W_\by^a := W_\by \cap \cU_{b,a}.
$$
The last main result of the paper is the following:
\begin{theorem}\label{thm:main-hyper}
For each $b=1,\ldots,\beta,$ and for each $a=1,\ldots,\aph_b$, the following 
statements hold true:

\smallskip\noindent 
(1) For each $i=0,\ldots,n-1,$ the functions $\by \mapsto \sgm_i^\infty
(W_\by^a)$ are continuous on $\cV_b$. Thus $\by \mapsto 
\sgm_i^\infty(W_\by)$ is continuous over $\Rs\setminus K(\vp)$.
 
\smallskip\noindent
(2) The function $\by \mapsto \Tht_n^\infty(\cW_\by^a)$ is also continuous on 
$\cV_b$. Thus $\by \mapsto \Tht_n^\infty(\cW_\by)$ is continuous over 
$\Rs\setminus  K(\vp)$.
\end{theorem}
\begin{proof} 
Let $\bc \in \Rs \setminus K(f)$. We can assume that both $\cW_\bc$ and $W_\bc$ 
are connected. Let us apply the Gauss-Bonnet formulas of Theorem 
\ref{thm:GaussBonnet} to the $n$-dimensional closed definable set $\cW_\by$.
%%viewed as the subset of $\cW_\bc\times \bc$ of $\Rn \times \Rs$. 
We get
\begin{equation}\label{eq:hyp-0}
\Lambda_0^\infty(\cW_\by) = \chi(\cW_\by) - 
\chi_n^\infty(\cW_\by) - \chi_{n-1}^\infty(\cW_\by).
\end{equation}
Moreover, for $k=0,\ldots,n-2$, 
%%let $k$ be any integer such that $k + \dim W - s \geq 0$.  
we have:
\begin{equation}\label{eq:hyp-k}
\Lbd_k^\infty(\cW_\by) = -
\chi_{n-k-1}^\infty(\cW_\by) + \chi_{n-k+1}^\infty(\cW_\by),
\end{equation}
and
\begin{equation}\label{eq:hyp-n-1}
\Lbd_{n-1}^\infty(\cW_\by) = \chi_2^\infty(\cW_\by)
\end{equation}
\begin{equation}\label{eq:hyp-n}
\Theta_n^\infty(\cY) = 
\Lbd_n^\infty(\cW_\by) = \chi_1^\infty(\cW_\by).
\end{equation}
Since $\cW_\by$ is $n$-dimensional, Equalities \eqref{eq:hyp-0}, 
\eqref{eq:hyp-k}, \eqref{eq:hyp-n-1}, and \eqref{eq:hyp-n} are non-trivial.

\medskip\noindent
Applying Proposition \ref{prop:cont-link-hyper}, we obtain that for $k = 0,
\ldots,n$, the functions
$$
\by \mapsto \Lbd_k^\infty(\cW_\by)
$$ 
are continuous on $\Rs \setminus K(\vp)$. It is enough to apply the 
expressions of Corollary \ref{cor:sgmi-infty} to conclude.
\end{proof}
%
%
%
%
%
%
%
%
%
%
%
%
%
%
%
%
%
%         **************************************************
%
%
%
%
%
%
%
%
%
%
%
%
%
%
%
%
%
%

\end{document}